\documentclass[11pt]{amsart}

\usepackage{amsfonts, amssymb, amscd}

\usepackage{verbatim}
\usepackage{amssymb}
\usepackage{mathrsfs}
\usepackage{graphicx}
\usepackage{bm}
\usepackage[all]{xy}
\usepackage{tikz}
\usepackage{subcaption}
\usepackage{subfiles}
\usepackage[toc,page]{appendix}
\usepackage{mathtools}
\usepackage{comment}
\usepackage{enumerate}
\usepackage{enumitem}

\usepackage{graphicx}
\graphicspath{ {images/} }

\usepackage{hyperref}

\numberwithin{equation}{section}

\newcommand{\Zz}{\mathbb{Z}}
\newcommand{\Cc}{\mathbb{C}}
\newcommand{\Pp}{\mathbb{P}}
\newcommand{\Rr}{\mathbb{R}}
\newcommand{\Qq}{\mathbb{Q}}
\newcommand{\Nn}{\mathbb{Z}_{>0}}

\newcommand{\Weil}{\operatorname{Weil}}

\newcommand{\Supp}{\operatorname{Supp}}

\newcommand{\lct}{\operatorname{lct}}
\newcommand{\Exc}{\operatorname{Exc}}

\newcommand{\Int}{\operatorname{relint}}

\newcommand{\Dd}{\mathcal{D}}
\newcommand{\Ff}{\mathcal{F}}

\newcommand{\Ii}{\mathcal{I}}

\newcommand{\Tt}{\mathcal{T}}

\newcommand{\Ss}{\mathcal{S}}

\newcommand{\BB}{\mathfrak{B}}

\newtheorem{theorem}{Theorem}[section]
\newtheorem{lemma}[theorem]{Lemma}
\newtheorem{proposition}[theorem]{Proposition}
\newtheorem{definition}[theorem]{Definition}

\newtheorem{corollary}[theorem]{Corollary}
\newtheorem{remark}[theorem]{Remark}

 \usepackage{todonotes}

\begin{document}

\title[ACC for LCT-polytopes]{ACC for log canonical threshold polytopes}



\begin{abstract}
We show that the log canonical threshold polytopes of varieties with log canonical singularities satisfy the ascending chain condition.
\end{abstract}

\author{Jingjun Han,Zhan Li, Lu Qi}

\address{Department of Mathematics, Johns Hopkins University, Baltimore, MD 21218, USA}
\email{jhan@math.jhu.edu}

\address{Beijing International Center for Mathematical Research, Peking University,
Beijing 100871, China} 
\email{hanjingjun@pku.edu.cn}

\address{Department of Mathematics, Southern University of Science and Technology, 1088 Xueyuan Rd, Shenzhen 518055, China} \email{lizhan@sustech.edu.cn}

	\address{
		Department of Mathematics, Princeton University,
		Princeton, NJ 08544, USA
	}
	\email{luq@princeton.edu}

\address{Beijing International Center for Mathematical Research, Peking University,
Beijing 100871, China} \email{lqipku@pku.edu.cn}

\maketitle

\pagestyle{myheadings}\markboth{\hfill  J.Han, Z.Li and L.Qi \hfill}{\hfill ACC for LCT-polytopes\hfill}

\tableofcontents

\section{Introduction}\label{sec: introduction}

Let $(X, \Delta)$ be a log canonical pair over $\Cc$, and let $D \geq 0$ be an $\Rr$-Cartier divisor on $X$. The log canonical threshold of $D$ with respect to $(X, \Delta)$ is defined by
\[
\lct(X, \Delta; D)\coloneqq\sup\{t \in \Rr \mid (X, \Delta + tD) \text{~is log canonical}\}.
\]
It can be viewed as a measurement for the complexity of the singularities. A conjecture of Shokurov \cite{Sho88, Sho92} predicts that in any fixed dimension, if the coefficients of $\Delta$ and $D$ belong to a set which satisfies the descending chain condition (DCC), then the set of all log canonical thresholds satisfies the ascending chain condition (ACC). This conjecture is settled in \cite{dFEM10, dFEM11} for varieties with bounded singularities (in particular for smooth varieties) and in \cite{HMX14} without any extra restriction (see \cite{Ale94, Pro01,Pro02, MP04, dFM09}, etc. for partial results towards this conjecture). \cite{dFEM10, dFEM11} and \cite{HMX14} use different methods where the former is more algebraic and the latter is more geometric. The (non-standard) algebraic method in \cite{dFM09} encodes a sequence of ideals in a single object. Their method was reinterpreted in a more algebraic geometric way and was called generic limits in \cite{Kol08}. We recommend \cite{Kol08} for a nice survey of log canonical thresholds and generic limits.

\medskip

The ACC for log canonical thresholds plays an important role in the recent development of birational geometry. For example, it is closely related to the termination of flips \cite{Sho04, Bir07}, birational boundedness of varieties of general type \cite{HMX14}, the existence of complements \cite{Bir16a} and the Birkar-Borisov-Alexeev-Borisov Theorem \cite{Bir16a,Bir16b}.

\medskip

In \cite{LM11}, as an analogue of the log canonical threshold of a single divisor, Libgober-Musta\c{t}\u{a} introduced the log canonical threshold polytope (LCT-polytope) for multiple divisors. To be precise, let $(X,\Delta)$ be a log canonical pair, and let $D_1$, $\ldots$, $D_s\geq0$ be $\mathbb{R}$-Cartier divisors (we call them testing divisors). The LCT-polytope $P(X, \Delta;D_1,\ldots,D_s)$ of $D_1,\ldots,D_s$ with respect to $(X,\Delta)$ is defined by
\begin{align*}
P(X, \Delta; &D_1,\ldots,D_s)\coloneqq\\
&\{(t_1,\ldots,t_s) \in \mathbb{R}_{\geq 0}^s\mid (X,\Delta+t_1 D_1+\ldots+t_s D_s)\text{~is log canonical}\}.
\end{align*}

\medskip

The main goal of this paper is to prove a generalization of the ACC for log canonical thresholds to the setting of log canonical threshold polytopes. We show that in any fixed dimension, for a fixed number of testing divisors, if the coefficients of all divisors involved belong to a DCC set, then the set of LCT-polytopes satisfies the ACC under the inclusion of polytopes.

\medskip

\begin{theorem}[ACC for LCT-polytopes]\label{thm: ACC for LCT-polytopes}
Let $n,s \in \Nn$ be positive integers, and let $\Ii \subseteq \Rr_{\ge0}$ be a DCC set. Let $\Ss$ be the set of all $(X, \Delta; D_1,\ldots, D_s)$, where
\begin{enumerate}
\item $X$ is a normal variety of dimension $n$,
\item $(X, \Delta)$ is log canonical, 
\item the coefficients of $\Delta$ belong to $\Ii$, and
\item $D_1, \ldots, D_s$ are $\Rr$-Cartier divisors, and the coefficients of $D_1,\ldots,D_s$ belong to $\Ii$.
\end{enumerate}
Then $\{P(X, \Delta; D_1, \ldots, D_s) \mid (X, \Delta; D_1, \ldots, D_s)\in \Ss\}$ is an ACC set.
\end{theorem}

\cite{LM11} proved Theorem \ref{thm: ACC for LCT-polytopes} when $X$ is smooth. Their method relies on generic limits. When $s=1$, Theorem \ref{thm: ACC for LCT-polytopes} is the ACC for log canonical thresholds, which is the main result of \cite{HMX14}.

\medskip

Theorem \ref{thm: ACC for LCT-polytopes} gives more information for the complexity of the singularities. Indeed, as is pointed out by \cite{LM11}, even if one is only interested in singularities of one divisor $D_1$, studying the LCT-polytopes $P(X, \Delta; D_1, \ldots, D_s)$ for various testing divisors reveals more information. 

\begin{figure}[ht]
	\begin{tikzpicture}
	
	\draw [help lines,<->] (0,2.5)--(0,0)--(3.2,0);
	\draw [dashed](0,2.1)--(1,2);
	\draw[dashed] (2.6,1)--(2.9,0);
	\draw[dashed] (2.3,1.4)--(2.6,1);
	\draw [dashed](2.3,1.4)--(2,1.7);
	\draw [dashed] (1,0)--(1,2.8);

	\draw[help lines] (2.4,0)--(2.4,2.5);
	\draw[help lines] (1.9,0)--(1.9,2.5);
	\draw[help lines] (1.5,0)--(1.5,2.5);
	\draw[help lines] (1.2,0)--(1.2,2.5);
	\draw[help lines] (1.1,0)--(1.1,2.5);
	\draw[help lines] (1.05,0)--(1.05,2.5);
	\draw[->](3,1.25)--(2.2,1.25);
	\draw[->](3,1.55)--(2,1.55);
	\draw[->](3,1.85)--(1.6,1.85);

	\draw [thick](1,2)--(2.6,1);
	\draw [thick] (1,2)--(2.3,1.4);
	\draw [thick] (1,2)--(2,1.7);
	
	\node [left] at (0.1,-0.1) {\footnotesize $\bf 0$};
	
	\node [right] at (3,1.25) {\tiny$P_1$};
	\node [right] at (3,1.55){\tiny$P_2$};
	\node [right] at (3,1.85){\tiny$P_3$};
	
	\node [below] at (0.9,2.1) {\footnotesize$\tau$};
	\node [below] at (.9,0.05) {\footnotesize$\tau_0$};
	\end{tikzpicture}
	\caption{}
	\label{fig: 0}
\end{figure}

It is natural to try to apply the results of \cite{HMX14} to prove Theorem \ref{thm: ACC for LCT-polytopes}. Suppose that we consider two testing divisors, and the dimensions of LCT-polytopes are equal to 2. \cite[Theorem 1.1]{HMX14} implies that any sequence $\{P_i\}_{i\in\Nn}$ of such polytopes stabilizes along vertical lines, horizontal lines, and lines passing through the origin. However, it is possible to construct a strictly increasing sequence of polytopes $\{P_i\}_{i\in\Nn}$ with the above properties, which does not stabilize near a common vertex $\tau$ (see Figure \ref{fig: 0}). Thus, it seems that Theorem \ref{thm: ACC for LCT-polytopes} does not follow from \cite[Theorem 1.1]{HMX14} directly even if $s=2$. In order to rule out a configuration as Figure \ref{fig: 0}, we use a two-step procedure by looking at a sequence of vertical lines approaching the unstable vertex $\tau$, which is illustrated in Section \ref{sec: sketch}.

\medskip

The general case is more complicated. We reduce Theorem \ref{thm: ACC for LCT-polytopes} to a local statement, and prove the ``ACC for local LCT-polytopes''. Recall that for a set $\Tt$ of linear functions defined on an interval $I$, there is a natural partial order $\succeq$, that is, for $f,g\in\Tt$, $f\succeq g$ if and only if $f(t)\ge g(t)$ for any $t\in I$.

\begin{theorem}[ACC for local LCT-polytopes] \label{thm: ACC for Local LCT-polytopes}
	Let $n,s \in \Nn$ be positive integers, and let $\Ii \subseteq \Rr_{\ge0}$ be a DCC set. Let $b_1,\ldots,b_s\ge0$ be non-negative real numbers, and let $a_1,\ldots,a_s$ be real numbers. Suppose that $\lambda_0>0$ is a positive real number, such that $a_{j}\lambda_0+b_{j} \geq 0$ for any $1\leq j\leq s-1$. 
	Then there is an ACC set $\Tt$ of linear functions satisfying the following.
	
	If $(X,\Delta;D_1,\ldots, D_s)$ satisfies that
		\begin{enumerate}
		\item $X$ is a variety of dimension $n$,
		\item $(X,\Delta)$ is log canonical, 
		\item the coefficients of $\Delta$ belong to $\Ii$, 
		\item $D_1,\ldots, D_s$ are $\Rr$-Cartier $\Rr$-divisors, and the coefficients of $D_1,\ldots, D_s$ belong to $\Ii$, and
		\item there exist a positive real number $c\le\lambda_0$ and a linear function $t(\lambda)$ of $\lambda$ defined on $[0,+\infty)$, such that $t(0)=b_s$, and for any $\lambda\in [0,c]$,
		$$0 \leq t(\lambda)=\sup\{ t \mid ({X}, \Delta+\sum_{j=1}^{s-1}(a_{j}\lambda+b_j)D_{j}+tD_{s}) \text{~is log canonical}\},$$
	\end{enumerate}
	then $t(\lambda)\in \mathcal{T}$.
\end{theorem}

Theorem \ref{thm: ACC for Local LCT-polytopes} follows from Theorem \ref{thm: linear global ACC}, the Global ACC for linear functional divisors. 
Divisors with linear function coefficients appears naturally in birational geometry. For example, \cite{Nak16,HLS19} considered that setting in the study of minimal log discrepancies.  

\medskip

Let $\Ff$ be a set of linear functions with real coefficients, and let $\Ff|_c\coloneqq \{f(c) \in \Rr \mid f(t) \in \Ff\}$ be the set of its values at $c\in \Rr$. Consider an $\Rr$-linear functional divisor $\Delta(t)= \sum f_i(t) D_i$, where $D_i$ are distinct prime divisors and $f_i(t) \in \Ff$. Denote this by $\Delta(t) \in \Ff$ for convenience. With the above conventions, we have the following result.

\begin{theorem}[Global ACC for linear functions]\label{thm: linear global ACC}
	Let $n \in \Nn$, and let $a<b$ be two real numbers. Let $\Ff$ be a set of real linear functions, such that for any $f(t)\in \Ff$, $f(t)\geq 0$ for $t\in[a,b]$, and $\Ff|_a\cup\Ff|_b$ is a DCC set. Then there exists a finite subset $\Ff'\subseteq \Ff$ satisfying the following. 
	
	If $(X,\Delta(t))$ is a log pair such that
	\begin{enumerate}
		\item $X$ is a normal projective variety of dimension $n$,
		\item $\Delta(t) \in \Ff$ is a linear functional divisor of $t$ on $[a,b]$,
		\item there exists $a<b_X \leq b$, such that $(X,\Delta(t))$ is log canonical for any $t\in[a, b_X]$, and
		\item $K_{X}+\Delta(t)$ is numerically trivial for any $t\in[a, b]$,
	\end{enumerate}
	then $\Delta(t) \in \Ff'$.
\end{theorem}

\begin{remark}	
If condition (3) in Theorem \ref{thm: linear global ACC} is replaced by ``there exists $a\leq a_X<b$, such that $(X,\Delta(t))$ is log canonical for any $t\in[a_X, b]$'', then Theorem \ref{thm: linear global ACC} still holds.
\end{remark}

If $\Ff\subset [0,1]$ is a set of constant functions, then the above theorem is the Global ACC Theorem, \cite[Theorem 1.5]{HMX14}. Our argument relies on \cite{HMX14}, so it does not give an independent proof of \cite[Theorem 1.5]{HMX14}.

\medskip

However, Theorem \ref{thm: ACC for Local LCT-polytopes} is not an ad hoc generalization of \cite[Theorem 1.5]{HMX14}. Indeed, there are cases where conditions of \cite[Theorem 1.5]{HMX14} are not satisfied but the coefficient set is still finite by the above result. For example, consider all the fixed dimensional log canonical pairs $(X, D_1)$ and $(X,(1-\frac{1}{m^2})D_1+\frac{1}{m}D_2)$, where $D_i$ are prime divisors, $m\in \Nn$, and $ K_X+D_1\equiv K_X+(1-\frac{1}{m^2})D_1+\frac{1}{m}D_2\equiv0$. We claim that there are only finitely many $m$ satisfying the above conditions. This does not follow from Global ACC since $\{\frac 1 m \mid m \in \Nn\}$ is not a DCC set. However, Theorem \ref{thm: linear global ACC} can be applied by considering the following data: $[a, b]=[0, 1]$, $b_X=\frac{1}{m}$, $\Ff=\{t\}\bigcup\cup_{m=1}^{+\infty}\{1-\frac{t}{m}\}$, and $\Delta(t)=(1-\frac{t}{m})D_1+t D_2$. Since $\Ff'$ is a finite set, $m$ belongs to a finite set.

\medskip

In the same fashion, we obtain an analogue of the ACC for the Fano spectrum (see \cite[Corollary 1.10]{HMX14}). 

\begin{corollary}\label{cor: Fano spectrum}
Let $n \in \Nn$, and let $a<b$ be two real numbers. Suppose that $\mathcal B\subseteq [a,b)$ is a DCC set. Let $\Ff$ be a set of real linear functions, such that for any $f(t)\in \Ff$, $f(t)\geq 0$ for $t\in [a,b]$, and $\Ff|_a\cup\Ff|_b$ is a DCC set. Let $\mathfrak{D}$ be the set of all $(X, \Delta(t),a_X)$, where
\begin{enumerate}

\item $X$ is a normal projective variety of dimension $n$,

\item $\Delta(t) \in \Ff$, $a_{X}\in \mathcal B$,

\item $(X,\Delta(t))$ is log canonical for $t\in[a_X, b]$,

\item $K_{X}+\Delta(b) \sim_\Rr 0$, and

\item $-(K_{X}+\Delta(a_{X}))$ is ample.

\end{enumerate}
Then the set 
\begin{align*}
\mathcal{R}\coloneqq
\{r \in \Rr \mid -(K_X+\Delta(a_X))\sim_{\Rr}rH &\text{ for some Cartier divisor } H,\\
&(X, \Delta(t),a_X)\in \mathfrak{D}\}
\end{align*}
is an ACC set.
\end{corollary}

The complex singularity exponent of a holomorphic function $f$, which is the analytic counterpart for the log canonical threshold, is the largest real number $c$ such that $|f|^{-c}$ is $L^2$-integrable. It appears in the study of K\"{a}hler-Einstein metrics \cite{Tia90}, and has been further studied by many mathematicians (see \cite{DK01, GZ15}, etc.). Related works had already appeared in the early 50's (see \cite{Kol08} for references). The LCT-polytope of multiple holomorphic functions $f_1, \ldots, f_s$ can be defined as
\[
\{(c_1, \ldots, c_s) \in \Rr^s_{\geq 0} \mid \prod_{i=1}^s |f_i|^{-c_i} \text{ is } L^2 \text{-integrable}\}.
\] 
It is expected that such polytopes also contain more information of analytic singularities.

\medskip

This paper is organized as follows. Section \ref{sec: sketch} illustrates the main ideas in the proof of Theorem \ref{thm: ACC for LCT-polytopes} for two testing divisors in the surface case. Section \ref{sec: preliminaries} gathers relevant definitions and preliminary results. Theorem \ref{thm: linear global ACC} is proved in Section \ref{sec: proof}. The remaining main results are proved in Section \ref{sec: proof of thm and cor}. 

\medskip

\noindent\textbf{Acknowledgements}.
This work originates from the Birational Geometry Seminar at BICMR, Peking University. We thank the participants for the seminar, especially Yifei Chen and Evgeny Mayanskiy. We thank Chenyang Xu, the advisor of the first and the third author, for his constant support, valuable discussions and suggestions. We also thank Chen Jiang for answering many of our questions, and Mircea Musta{\c t}{\u a} and Vyacheslav Shokurov for their interests in the work. J. H. thanks
his advisor Gang Tian for constant support and encouragement. The authors sincerely thank referees for their long list of suggestions, comments, and corrections which improve
the paper considerably. A simplification of the argument in Proposition \ref{prop:not bounded} is due to their suggestions. This work is partially supported by NSFC Grant No.11601015 and the Postdoctoral Grant No.2016M591000.

\section{The main ideas in the proof}\label{sec: sketch}
 
In order to orient the reader, we give a sketch of the proof of the ACC for LCT-polytopes for surfaces with two testing divisors.

\medskip

Consider the set of triples $(X_i; D_{i,1}, D_{i,2})$ satisfying the assumptions of Theorem \ref{thm: ACC for LCT-polytopes} with $\dim X_i =2$. Denote $P_i \coloneqq P(X_i; D_{i,1}, D_{i,2})$, where
\begin{equation*}
P(X_i; D_{i,1}, D_{i,2})=\{(x_1,x_2)\in\Rr^2_{\geq 0}\mid(X_i,x_1D_{i,1}+x_2D_{i,2}){\rm~is~log~canonical}\}.
\end{equation*}
Suppose that $\dim P_i=2$ and $P_1\subsetneq P_2\subsetneq P_3\subsetneq\cdots$ is a strictly increasing sequence. 

\medskip

By Lemma \ref{le: DCC implies unstable point}, there exists a point $\beta\in P_i$ for all $i\gg 1$, such that for any open set $U\ni \beta$, the set $\{P_i\cap U\mid i\in \Zz_{>0}\}$ contains infinitely many elements. Possibly passing to a subsequence, we can show that for any $i$, there exists a facet (a line segment) $F_i\ni\beta$ of $P_i$ satisfying the following:
\begin{enumerate}
    \item if $H_i\subseteq\Rr^2$ is the hyperplane (hence a line) containing $F_i$, then $H_i\notin\{H_1,\ldots,H_{i-1}\}$,
    \item $H_i$ is not parallel to the $x_2$-axis, and 
   \item if $\pi:\Rr^2\to\Rr$ is the projection $(x_1,x_2)\mapsto x_1$, then there exists a ray $R$ with vertex $\pi(\beta)$, such that $R \subseteq \pi(H_i)$, and $\pi(F_i) \cap R \neq \emptyset$ for $i\gg 1$.
\end{enumerate}
Figure \ref{fig: 1} gives an illustration. For more details, see Section \ref{sec: proof of thm and cor}.

\begin{figure}[ht]
\begin{tikzpicture}

\draw [help lines,<->] (0,2.5)--(0,0)--(3.2,0);
\draw [dashed](0,2.1)--(1,2);
\draw[dashed] (2.6,1)--(2.9,0);
\draw[dashed] (2.3,1.4)--(2.6,1);
\draw [dashed](2.3,1.4)--(2,1.7);
\draw [dashed] (1,0)--(1,2.8);
\draw[->](1,0)--(3.2,0);
\draw[help lines] (2.4,0)--(2.4,2.5);
\draw[help lines] (1.9,0)--(1.9,2.5);
\draw[help lines] (1.5,0)--(1.5,2.5);
\draw[help lines] (1.2,0)--(1.2,2.5);
\draw[help lines] (1.1,0)--(1.1,2.5);
\draw[help lines] (1.05,0)--(1.05,2.5);
\draw[->](3,1.25)--(2.2,1.25);
\draw[->](3,1.55)--(2,1.55);
\draw[->](3,1.85)--(1.6,1.85);

\draw [thick](1,2)--(2.6,1);
\draw [thick] (1,2)--(2.3,1.4);
\draw [thick] (1,2)--(2,1.7);
\draw [->, thick] (1,0)--(3,0);

\node [left] at (0.1,-0.1) {\footnotesize $\bf 0$};

\node [right] at (3,1.25) {\tiny$F_1$};
\node [right] at (3,1.55){\tiny$F_2$};
\node [right] at (3,1.85){\tiny$F_3$};

\node[below] at (1.9,0) {\footnotesize$\lambda$};
\node [below] at (0.9,2.1) {\footnotesize$\beta$};
\node [below] at (.9,0.05){\footnotesize$\pi(\beta)$};
\node[below] at (3,0){\footnotesize$R$};
\node[right] at (3.2,0){\footnotesize$x_1$};
\node[above] at (0,2.4){\footnotesize$x_2$};

\end{tikzpicture}
\caption{}
\label{fig: 1}
\end{figure}

Let $a=\pi(\beta)$. If $x_2=t_i(x_1)$ is the equation of $H_i$, then $\{t_i(\lambda)\}_{i\in\Zz_{>0}}$ is strictly increasing for any $\lambda>a$. For each $i$, by definition, there exists $b_{X_i}>a$ such that 
$t_i(\lambda)=\sup\{x_2 \mid (\lambda,x_2)\in P_i\}=\sup\{x_2\ge0 \mid (X_i, \lambda D_{i,1}+ x_2D_{i,2}) {\rm~is~log~canonical}\}$ for any $\lambda\in[a,b_{X_i}]$. Let $b=b_{X_1}$, and $\Delta_i(\lambda) \coloneqq \lambda D_{i,1} + t_i(\lambda)D_{i,2}$. If some irreducible component $E_i$ of $D_{i,2}$ is a log canonical center of $(X_i,\Delta_i(\lambda))$ for infinitely many $i$, then for a fixed $\lambda\in(a,b)$, we have 
\[
\lambda{\rm mult}_{E_i}(D_{i,1})+t_i(\lambda){\rm mult}_{E_i}(D_{i,2})=1,
\]
with ${\rm mult}_{E_i}(D_{i,2})\neq 0$. Thus $t_i(\lambda)$ belongs to an ACC set, a contradiction.  

\medskip

Hence we may assume that for any $i$, the coefficient of any irreducible component of $D_{i,2}$ in $\Delta_i(\lambda)$ is less than 1. For fixed $i$, we can show that there exists a common dlt modification (Proposition \ref{prop: common dlt}) $\phi_i: Y_i \to X_i$ for any $\lambda\in (a,b_{X_i})$, such that
\begin{equation}\label{eq: surface case}
K_{Y_i}+E_i+\tilde\Delta_i(\lambda)=\phi_i^{*}(K_{X_i}+\Delta_i(\lambda)),
\end{equation} 
where $\tilde\Delta_i(\lambda)$ is the strict transform of $\Delta_i(\lambda)$. Here $E_i$ is a reduced exceptional divisor, and it intersects the strict transform of $D_{i,2}$. For simplicity, we assume that $E_i$ is a prime divisor.

\medskip

Restricting \eqref{eq: surface case} to $E_i$, we get
\begin{equation}\label{eq: num trivial}
K_{E_i} + \Theta_i(\lambda):=(K_{Y_i}+E_i+\tilde\Delta_i(\lambda))|_{E_i}\equiv0.
\end{equation} 

By the adjunction formula (see Proposition \ref{prop: adjunction}), the coefficients of $\Theta_i(\lambda)$ are of the form 
\begin{equation}\label{eq: m}
\frac{m-1+a_i\lambda +b_it_i(\lambda)}{m}
\end{equation}
where $m\in\Nn$ and $a_i,b_i\in\sum\Ii$ (see Lemma \ref{le: dcc} for the notation). For simplicity, we only consider the case $m=1$. Since $E_i$ is a smooth rational curve and $\deg K_{E_i}=-2$, \eqref{eq: num trivial} implies
\begin{equation*}\label{eq: compare coefficients}
\sum_{j=1}^{l_i} (a_i^j\lambda + b_i^jt_i(\lambda)) = 2
\end{equation*}
for some $l_i\in\Zz_{>0}$. Note that at least one $b_i^j\neq 0$ as $E_i$ intersects $D_{i,2}$. Since $a_i^j,b_i^j$ belong to the DCC set $\sum\Ii$, and $\{t_i(\lambda)\}_{i\in\Zz_{>0}}$ is strictly increasing for any fixed $\lambda>a$, $\{l_i\}$ is bounded. Passing to a subsequence, we may assume that $l_i=l$, and that both $\{a_i^j\}_{i\in\Nn}$ and $\{b_i^j\}_{i\in\Nn}$ are increasing for each $j$. Then
\begin{equation*}\label{eq: compare terms}
\sum_{j=1}^{l} (a^j_{1}\lambda +  b^j_{1}t_{1}(\lambda)) = 2 \text{~and~} \sum_{j=1}^{l} (a^j_{2}\lambda  + b^j_{2}t_{2}(\lambda)) = 2
\end{equation*} 
imply that $t_{1}(\lambda) \geq t_{2}(\lambda)$ for any $\lambda>a$, which is a contradiction since $\{t_i(\lambda)\}_{i \in \Nn}$ is strictly increasing.

\medskip

There are three major simplifications in the above argument which correspond to three hurdles to overcome in the general case:

\medskip

First, we assumed that there are only two testing divisors. Hence the ``bad'' configuration of $\{P_i\}$ is simple (see Figure \ref{fig: 1}). 
For multiple testing divisors, the configurations can be much more complicated, see Section \ref{sec: proof of thm and cor}. 

\medskip

Second, we assumed $m=1$ in \eqref{eq: m}. In general, the case where $m$ is unbounded needs to be ruled out, see Proposition \ref{prop: bounded implies thm P}. 

\medskip

Third, in the above, we only dealt with surfaces. In higher dimensional cases, one can still apply the adjunction formula to some common log canonical places and obtain numerically trivial pairs. However, the exceptional divisors $E_i$ may not belong to a bounded family. We use the minimal model program (MMP) to reduce to the Fano case. 
If all the $E_i$ are $\epsilon$-lc for some $\epsilon>0$, then by Birkar-Borisov-Alexeev-Borisov Theorem (see Theorem \ref{thm: BAB}), they belong to a bounded family, 
and a similar argument as above works. Otherwise, we play the two ray game to create log canonical centers, see Proposition \ref{prop:not bounded}.

\section{Preliminaries}\label{sec: preliminaries}

\subsection{Arithmetic of sets}

\begin{definition}[DCC and ACC set]\label{def: DCC and ACC}
	A partially ordered set $(\Ii,\succeq)$ satisfies the \emph{descending chain condition} (resp. \emph{ascending chain condition}), if for any descending sequence $a_1\succeq a_2\succeq\cdots$(resp. ascending sequence $a_1\preceq a_2\preceq\cdots$) in $\Ii$, there exists $n \in  \Nn$, such that $a_i=a_j$ for any $i,j\geq n$. For simplicity, we write DCC for descending chain condition and ACC for ascending chain condition.
\end{definition}

\begin{remark}
This paper mainly deals with three kinds of partially ordered sets. The set of polytopes is ordered by inclusion $\subseteq$. The set of real numbers $\Rr$ is ordered by the usual $\leq$. Besides, for a set $\Ff$ of real-valued functions defined on a closed interval $[a,b]$ where $a<b$, the partial order $\succeq$ is defined pointwisely, that is, for $f,g\in\Ff$, $f\succeq g$ if and only if $f(t)\geq g(t)$ for any $t\in[a,b]$. 
\end{remark}

The proof of the following result is elementary and we omit it.

\begin{lemma}\label{le: dcc}
	DCC sets have following three properties.
	\begin{enumerate}
		\item If $\Ii$ is a totally ordered DCC set, then there exists a unique minimal element of $\Ii$. Thus, for any sequence $\{a_i\}_{i \in \Nn}$ with $a_i\in \Ii$, there exists an increasing subsequence $\{a_{i_j}\}_{j\in \Nn}$.
		
		\item If $\Ii_i, i=1,\ldots,n$ are DCC subsets of a partially ordered set $\Ii$, then $\bigcup_{i=1}^n \Ii_i$ is a DCC set.
		
	    \item If $\Ii \subseteq  \Rr_{\ge 0}$ is a DCC set, then
\[
\sum \Ii\coloneqq\{\sum_{i_l \in \Ii} i_l \in \Rr \text{ a~finite~sum}\},
\]
is a DCC set.
	\end{enumerate}
\end{lemma}

\subsection{Log canonical thresholds and LCT-polytopes}\label{subsection: Log canonical threshold and LCT-polytopes}

For definitions of log canonical (lc), kawamata log terminal (klt) and divisorial log terminal (dlt) singularities and their basic property, see \cite[\S 2]{KM98}.

\medskip

In the sequel, a log pair $(X,\Delta)$ consists of a normal variety $X$ and an $\Rr$-divisor $\Delta \geq 0$ such that $K_X+\Delta$ is $\Rr$-Cartier. If the coefficients of $\Delta$ are at most 1, then $\Delta$ is called a boundary. 

\medskip

For a birational morphism $f: Y \to X$ with $Y$ a normal variety, write
\begin{equation*}
K_Y \sim_\Rr f^*(K_X + \Delta) + \sum_E a(E, X, \Delta) E,
\end{equation*} 
where the irreducible divisors $E$ run over the components of exceptional divisors $\Exc(f)$ and the strict transform $f_*^{-1}(\Delta)$ of $\Delta$. The real number $a(E, X, \Delta)$ is called the \emph{discrepancy} of $E$ with respect to $(X, \Delta)$, and it is independent on $Y$. For a real number $\epsilon \geq 0$, $(X, \Delta)$ is called \emph{$\epsilon$-lc} if the discrepancy $a(E, X, \Delta)\geq -1 + \epsilon$ for any divisor $E$ over $X$.

\begin{definition}[Lc place, lc center]
Let $(X,\Delta)$ be an lc pair. An \emph{lc place} of $(X,\Delta)$ is a prime divisor $E$ over $X$ such that $a(E,X,\Delta)=-1$. An \emph{lc center} is the image of an lc place on $X$.
\end{definition}

\begin{definition}[Log canonical threshold]
Let $(X,\Delta)$ be an lc pair and let $D$ be an effective $\mathbb{R}$-Cartier $\Rr$-divisor. The \emph{log canonical threshold} of $D$ with respect to $(X,\Delta)$ is
\begin{equation*}
\lct(X,\Delta;D)\coloneqq {\rm sup}\{t \in\mathbb{R}\mid (X,\Delta+tD) \text{~is log canonical}\}.
\end{equation*}
\end{definition}

For $I\subseteq [0,1]$, we write $\Delta \in I$ if the coefficients of the $\Rr$-divisor $\Delta=\sum c_i \Delta_i$ belong to $I$. Let
\begin{equation*}
\mathcal{T}_n(I)\coloneqq \{(X,\Delta)\mid (X, \Delta) {\rm~is~log~canonical}, \dim X=n, \Delta\in I\}.
\end{equation*}
Suppose $J\subseteq \Rr_{\geq 0}$, let
\begin{equation*}
{\rm LCT}_n(I,J)\coloneqq\{\mathrm{lct}(X,\Delta;M)\mid (X,\Delta)\in\mathcal{T}_n(I), M \in J\}
\end{equation*}
be a subset of real numbers. Under the above notation, we have the following important result on log canonical thresholds.

\begin{theorem}[ACC for log canonical thresholds, {\cite[Theorem 1.1]{HMX14}}]\label{thm: ACC for lct}
Fix a positive integer $n$, $I\subseteq [0,1]$ and a subset $J$ of the positive real numbers.

If $I$ and $J$ satisfy the DCC, then ${\rm LCT}_n(I,J)$ satisfies the ACC.
\end{theorem}

Log canonical threshold can be generalized to the case of multiple divisors (called \emph{testing divisors}) and we get LCT-polytopes (see \cite{LM11}).

\begin{definition}[LCT-polytope]\label{def: LCT-polytope}
Let $(X,\Delta)$ be an lc pair. Let $D_1, \ldots, D_s>0$ be $\mathbb{R}$-Cartier $\mathbb{R}$-divisors. The \emph{LCT-polytope} $P(X, \Delta;D_1,\ldots,D_s)$ of $D_1,\-\ldots,D_s$ with respect to $(X,\Delta)$ is the set
\begin{equation*}
\{(t_1,\ldots,t_s) \in \mathbb{R}_{\geq 0}^s \mid (X,\Delta+t_1 D_1+\ldots+t_s D_s)\text{~is log canonical}\}.
\end{equation*}
\end{definition}

Since the boundaries of an LCT-polytope are defined by linear equations whose coefficients are determined by a log resolution of $(X, \Delta+\sum_{i} D_i)$, an LCT-polytope is indeed a bounded polytope in $\Rr^s_{\geq 0}$ (see \cite[Lemma 2.3]{LM11} or Lemma \ref{lem: LCTfacet}). 

\medskip

Let $P\subseteq \Rr^s$ be an LCT-polytope of dimension $s$, a facet (that is, a face of dimension $s-1$) $F$ of $P$ is called an \emph{LCT facet} if $F\nsubseteq\cup_{i=1}^{s}\{(x_1,\ldots,x_s) \mid x_i=0\}$.

\begin{lemma}\label{lem: LCTfacet}
	Let $P=P(X,\Delta; D_{1},\ldots,D_{s}) \subseteq{\Rr^s}$ be an LCT-polytope of dimension $s$. Suppose that $F$ is an LCT facet of $P$, and $H$ is the hyperplane containing $F$. There exist $a_1,\ldots,a_s\ge0$ such that the connected component of $\Rr^s \backslash H$ which contains $P$ is 
	$$\{(x_1,\ldots,x_s)\mid\sum_{i=1}^s a_ix_i\le a_0\}.$$
	In particular, $P$ is a bounded convex polytope.
\end{lemma}
\begin{proof}
    Let $\Delta=\sum b_{0,j}B_j$, and $D_i=\sum b_{i,j}B_j$, where $B_j$ are distinct prime divisors. Let $f:Y\to X$ be a log resolution of $(X,\Supp(\Delta+\sum_{i=1}^s D_i))$, then
	$$K_Y+\tilde{\Delta}+\sum_{i=1}^sx_i\tilde{D_i}+\sum_E (\sum_{i=1}^s a_{E,i}x_i+b_{E})E=f^{*}(K_X+\Delta+\sum_{i=1}^sx_iD_i),$$
	where $\tilde{\Delta}, \tilde{D_i}$ are strict transforms of $\Delta, D_i$ on $Y$, $E$ runs over all exceptional divisors, and $a_{i,E}\ge0$. Thus 
	\begin{align*}
	P=&\Rr_{\geq 0}^s \bigcap (\cap_E\{(x_1,\ldots,x_s)\mid\sum_{i=1}^s a_{E,i}x_i+b_{E} \leq 1\})\\
	&\cap_{j} \{(x_1,\ldots,x_s)\mid\sum_{i=1}^s b_{i,j}x_i+b_{0,j}\leq 1\}. 
	\end{align*}
	
    The connected component of $\Rr^s \backslash H$ which contains $P$ is  
	$$\{(x_1,\ldots,x_s)\mid\sum_{i=1}^s a_ix_i+b\le 1\},$$
	where either $a_i=a_{E,i}\ge 0$, $1 \leq i \leq s$ and $b=b_E$ for some $E$ or $a_i=b_{i,j}\ge0 $, $1 \leq i \leq s$ and $b=b_{0,j}$ for some $j$. Since $\bm{0}_s:=(0,0,\ldots,0)\in P$, $a_0\coloneqq 1-b\ge \sum_{i=1}^s a_i\cdot 0=0$.	
\end{proof}

Note that the interior of $P(X, \Delta;D_1,\ldots,D_s)$ does not correspond to the log canonical thresholds of $\{D_i\}$. However, we abuse the language and call the whole polytope the LCT-polytope.

\begin{lemma}\label{le: compactness of union}
Let $\{P_i\coloneqq P(X_i,\Delta_i;D_{i,1},\ldots,D_{i,s})\}$ be an increasing sequence of LCT-polytopes of dimension $s$. 
\begin{enumerate}
	\item If along any ray $R$ starting from the origin, this sequence of polytopes stabilizes (that is, $\{R \cap P_i\}$ stabilizes), then the union $\cup_{i\geq 1} P_i$ is a closed set.
	\item If $X_i$ $(i \in \Nn)$ are of the same dimension and the coefficients of $\Delta_i$ and $D_{i,j}$ $(1 \leq j \leq s)$ belong to a DCC set $\Ii$, then the union $\cup_{i\geq 1} P_i$ is a compact set.
\end{enumerate}
\end{lemma}

\begin{proof}
 For (1), let $(x_1,\ldots,x_s)\in \Rr_{\geq 0}^s $ be any point such that $(x_1,\ldots,x_s) \notin \cup_{i\geq 1} P_i$. Let $R$ be the ray from the origin $\bm{0}_s$ to $(x_1,\ldots,x_s)$. If $(rx_1,\ldots,rx_s)$ is the stable point along $R$, then $r<1$. Consider the open set
 \begin{equation*}
 T=\{(y_1,\ldots,y_s)\in \Rr^s \mid y_i>rx_i,i=1,2,\ldots,s\}.
 \end{equation*}
 Then $(x_1,\ldots,x_s) \in T$. We claim that $T \cap P_k = \emptyset$ for all $k$. Suppose to the contrary that there exists $(a_1, \ldots, a_s) \in T \cap P_k$ for some $k$. Let $r' = \min\{\frac{a_i}{rx_i}\}$. Then $a_i\ge r'rx_i$ for any $i$, and $(r'rx_1, \ldots, r'rx_s) \in P_k$. Since $a_i > rx_i$ for any $i$, $r'>1$. This contradicts the stability of $(rx_1,\ldots,rx_s)$. Thus, $(x_1,\ldots,x_s)$ does not belong to the closure of $\cup_{i\geq 1} P_i$.

\medskip

 For (2), if $c>0$ is the minimal nonzero element of $\Ii$, then $P_k\subseteq [0, 1/c]^s$ is bounded. Let $R \coloneqq \{(\alpha_1 t, \ldots, \alpha_s t) \mid t \in \Rr_{\geq 0}\}$ be a ray starting from the origin, where $\alpha_i\ge0$ and $(\alpha_1,\ldots,\alpha_s)\neq \bm{0}_s$. Let $(\alpha_1 t_k, \ldots, \alpha_s t_k)$ be the intersection point of $R$ with an LCT facet of $P_k$, then we have
\begin{equation*}
t_k = \lct(X_k, \Delta_k; \alpha_1 D_{k,1}+ \ldots+ \alpha_s D_{k,s}).
\end{equation*}
 The coefficients of $\alpha_1 D_{k,1}+ \ldots+ \alpha_s D_{k,s}$ belong to $\sum_{i=1}^s \alpha_i \Ii$, which is a DCC set by Lemma \ref{le: dcc}. By Theorem \ref{thm: ACC for lct}, $\{P_i\}$ stabilizes along $R$. Thus $\cup_{i\geq 1} P_i$ is closed by (1).
\end{proof}

\begin{definition}
	Let $X$ be a normal variety, and let $\Delta_i$ be distinct prime divisors. If $f_i(t):[a,b]\to\Rr$ is an $\Rr$-linear function (resp. piecewise $\Rr$-linear function), then we call the formal finite sum $\Delta(t)=\sum_{i} f_i(t)\Delta_i$ an \emph{$\Rr$-linear functional divisor} (resp. \emph{piecewise $\Rr$-linear functional divisor}) of $t$ on $[a,b]$. We also call it a linear functional divisor (resp. piecewise linear functional divisor) of $t$ on $[a,b]$ for simplicity.
	
	The support of $\Delta(t)$ is defined by $\Supp\Delta(t)\coloneqq\Supp \sum_{\{i\mid f_i(t)\neq 0\}}\Delta_i$. 
	\end{definition}

\begin{lemma}\label{le: linearity of lct}
	Let $a<b$ be two real numbers. Let $X$ be a normal variety, $\Delta(t)$ a piecewise linear functional divisor of $t$ on $[a,b]$, and $D > 0$ an $\Rr$-Cartier $\Rr$-divisor. Suppose that $(X,\Delta(t))$ is lc for $t\in[a,b]$. Then 
	$$\zeta(t)\coloneqq\sup\{\tau\in\Rr\mid(X, \Delta(t)+\tau D)\text{ is log canonical}\}$$ is a non-negative piecewise linear function of $t$ on $[a,b]$.
	
	Moreover, if $\Delta(t)$ is a linear functional divisor and $\zeta(t)$ is a linear function of $t$ on $[a,a+\epsilon_0]$ for some $0<\epsilon_0< b-a$, then an lc place of $(X,\Delta(t)+\zeta(t)D)$ for some $t\in(a,a+\epsilon_0)$ is an lc place for any $t\in[a,a+\epsilon_0]$.
\end{lemma}

\begin{proof}
	We first show that $\zeta(t)$ is a non-negative piecewise function. It suffices to show the case when $\Delta(t)$ is a linear functional divisor, and apply the conclusion to a subdivision of the interval $[a,b]$. 
		
	\medskip
	
    Take a log resolution $\pi:W\to X$ of $(X,\Supp\Delta(t)+\Supp D)$ and denote the set of exceptional divisors by $\{E_k\}$. We have
	\begin{equation*}\label{eqn:double linear resolution}
		K_W+\tilde\Delta(t)+\tau\tilde D=\pi^*(K_X+\Delta(t)+\tau D)+
	\sum_k a_k(t,\tau)E_k,
	\end{equation*}
	where $\tilde\Delta(t)$ and $\tilde D$ are the strict transforms of $\Delta(t)$ and $D$ respectively, and $a_k(t,\tau)\coloneqq a(E_k,X,\Delta(t)+\tau D)$ is a bilinear function of $t$ and $\tau$. By assumption, $a_k(t,0)\geq -1$ for any $t\in[a,b]$.
	
	\medskip
	
	For $t\in[a,b]$ and $k\ge1$, define $\tau_k(t)\coloneqq\sup\{\tau\mid a_k(t,\tau)\geq -1\}$. Then either $\tau_k(t) \equiv +\infty$ or $\tau_k(t)$ is a non-negative linear function of $t$ on $[a,b]$. By definition, $\zeta(t)=\min_k\{\tau_k(t)\}$ is a piecewise linear function of $t$ on $[a,b]$.
		 
	 \medskip

	 Suppose that $\Delta(t)$ is a linear functional divisor and $\zeta(t)$ is a linear function of $t$ on $[a,a+\epsilon_0]$, then $a_k(t,\zeta(t))$ is a linear function of $t$ on $[a,a+\epsilon_0]$ for each $k$. Thus if $a_k(t,\zeta(t))=-1$ for some $t\in(a,a+\epsilon_0)$, then $a_k(t,\zeta(t))=-1$ for any $t\in[a,a+\epsilon_0]$. \end{proof}

\subsection{Adjunction}
For a dlt pair $(X, \Delta=S+B)$ where $S$ is a prime divisor, by the adjunction formula, there exists a dlt pair $(S,{\rm Diff}_S(B))$, such that $(K_X+\Delta)|_{S} = K_S + {\rm Diff}_S(B)$. Moreover, if $B= \sum_{i=1}^s b_i B_i$, then the coefficients of ${\rm Diff}_S(B)$ are of the form $\frac{m-1+\sum_{i=1}^s r_ib_{i}}{m}\leq1$, where $m$ is a positive integer, and $r_i$ are non-negative integers. In this paper, we need an adjunction formula for linear functional divisors.

\begin{definition}\label{def: coefficient set}
Let $a<b$ be two real numbers, and let $\mathcal{F}$ be a set of non-negative linear functions defined on $[a, b]$. For any $a<c\le b$, define
\begin{equation*}
(\mathcal{F}_{+},[a,c])\coloneqq\{v(t) \mid v(t)=\sum_{f_i(t)\in \mathcal{F}}n_if_i(t) \leq 1 \text{~for~}t\in[a,c], n_i\in \Nn\}\bigcup\{0\},
\end{equation*}
where the sum $\sum$ stands for a finite sum, and
\begin{equation*}
\begin{split}
D(\mathcal{F},[a, c])\coloneqq\{w(t) \mid &w(t)=\frac{m-1+v(t)}{m} \leq 1 \text{~for~}t\in[a,c], \\
&m\in\Zz_{> 0},v(t)\in(\mathcal{F}_{+}, [a,c])\}.
\end{split}
\end{equation*}

In addition, define
\begin{equation*}
\Dd(\Ff, [a,b])\coloneqq\bigcup_{b\ge c>a}D(\Ff, [a, c]).
\end{equation*}
If $D_0\subseteq \Dd(\Ff)$, then define a subset of $\Ff$ by
\begin{equation*}
\begin{aligned}
\Dd^{-1}(\Ff,D_0)\coloneqq \{f(t) \in \Ff \mid &\frac{m-1+nf(t)+\sum n_jf_j(t)}{m} \in D_0,\\
& m,n,n_j\in\Zz_{>0}, f_j(t)\in\Ff\}.
\end{aligned}
\end{equation*}
\end{definition}
\begin{remark}
When $c=b$ and the interval $[a,b]$ is clear from the context, we write $\mathcal{F}_{+}$ and $D(\Ff)$ for $(\mathcal{F}_{+},[a, b])$ and $D(\mathcal{F},[a, b])$  respectively.

In the proof of Theorem \ref{thm: linear global ACC}, we apply Lemma \ref{le: linearity of lct} and shrink the interval $[a,b_X]$. This is the reason to introduce $\Dd(\Ff)$. \end{remark}

\begin{proposition}[Adjunction formula for linear functional divisors]\label{prop: adjunction}Let $a<b$ be two real numbers. Let $X$ be a normal variety, $S$ a prime divisor on $X$, and $\Gamma(t)=\sum_{i=1}^s b_i(t)B_i$ a linear functional divisor of $t$ on $[a,b]$. Suppose that $(X, \Gamma(t)+S)$ is dlt for $t\in[a,b]$. Then there is a linear functional divisor $\Theta(t)$ on $S$, such that
\begin{equation*}
(K_X+S+\Gamma(t))|_{S}=K_{S}+\Theta(t) \text{~for~} t\in[a,b],
\end{equation*}
$(S, \Theta(t))$ is dlt for $t\in[a,b]$, and the coefficients of $\Theta(t)$ belong to
\begin{equation*}
D(\{b_1(t),\ldots,b_m(t)\},[a,b]).
\end{equation*}
\end{proposition}

\begin{proof}	
For any fixed $c \in (a, b)$, by \cite[\S 3]{Sho92} or \cite[\S 16]{Fli92}, there is an $\Rr$-divisor $\Theta(c)$ on $S$, such that
	\begin{equation*}
	(K_X+S+\Gamma(c))|_{S}=K_{S}+\Theta(c),
	\end{equation*}
	$(S, \Theta(c))$ is dlt, and the coefficient of $\Theta(c)$ at a codimension 1 point $V$ of $S$ is $\frac{m-1+\sum r_ib_i(c)}{m}$, where $m\in\Zz_{>0}$ is the order of the cyclic group $\Weil(\mathcal{O}_{X,V})$, and $r_i\in\Zz_{\ge0}$ is the coefficient of $mB_i|_{S}$ at $V$ for $1\le i\le s$. Hence $m$ and $r_i$ are independent of the choice of $c$, and the coefficient of $\Theta(t)$ at $V$ is 
	\[
	\frac{m-1+\sum r_ib_i(t)}{m}\in D(\{b_1(t),\ldots,b_m(t)\},[a,b]).
	\]
\end{proof}

The following lemma shows that the set of linear functions $\mathcal{D}(\Ff)$ behaves well after adjunction.

\begin{lemma}\label{le: properties of derived set}
Let $a<b$ be two real numbers. Let $\mathcal{F}\ni 1$ be a set of non-negative linear functions defined on $[a,b]$, and $D_0$ a finite subset of $\mathcal{D}(\Ff)$. Then
	\begin{enumerate}
		\item $D(\Ff)_+=D(\Ff)$,
		\item $\mathcal{D}(\mathcal{D}(\mathcal{F}))=\mathcal{D}(\mathcal{F})$,
		\item if $\Ff|_{a}\cup\Ff|_{b}$ is a DCC set, then  $\mathcal{D}^{-1}(\Dd(\Ff),D_0)$ is a finite set.
	\end{enumerate}
\end{lemma}

\begin{proof}
	The proof of (1) and (2) is similar to that of \cite[Lemma 4.4]{MP04}.
	
	\medskip
	
	\noindent(1) Suppose $g(t)\in D(\Ff)_+$. Then by definition there exist linear functions $f_i(t)\in \Ff_+$ and $m_i\in\Nn$ such that
	\begin{equation*}
	g(t)=\sum_i(\frac{m_i-1}{m_i}+\frac{f_i(t)}{m_i}).
	\end{equation*}
	Suppose that there are more than one index $i$ such that $m_i\geq 2$, say $m_1\ge2$, $m_2\ge2$. Since $g(t)\leq 1$, we know that $m_1=m_2=2$, $f_i(t)=0$, hence $g(t)=1$ and we are done. Otherwise, $g(t)$ is of the form $g(t)=\frac{m-1+f(t)}{m}$ where $f(t)\in\Ff_+$, thus $g(t)\in D(\Ff)$.
	
    \medskip 
  
	\noindent(2) By definition $\Dd(\Ff)\subseteq \Dd(\Dd(\Ff))$. To prove the inverse inclusion, suppose $h(t)\in \Dd(\Dd(\Ff))$. By definition, there exist a linear function $g(t)\in \Dd(\Ff,[a,c'])_+$ and an $n\in\Nn$ such that
	\begin{equation*}
	h(t)=\frac{n-1}{n}+\frac{g(t)}{n}.
	\end{equation*}
	Write $g(t)=\sum_i g_i(t)$ where $g_i(t)\in\Dd(\Ff,[a,c'])$, then there exists $c_i\in (a,c']$ such that 
	$g_i(t)\in D(\Ff,[a,c_i])$. 
	
	\medskip
	
	Set $c\coloneqq \min\{c_i\}$, then 
	$g(t)\in D(\Ff,[a,c])_+$. Apply (1) to $[a,c]$, and we know that $g(t)\in D(\Ff,[a,c])_+=D(\Ff,[a,c])$. So there exist a linear function $f_{+}(t)\in (\Ff_+,[a,c])$ and an $m\in\Nn$ such that
	\begin{equation*}
	g(t)=\frac{m-1}{m}+\frac{f_{+}(t)}{m}.
	\end{equation*}
	Thus
	\begin{equation*}
	\begin{aligned}
	h(t)=\frac{n-1}{n}+\frac{g(t)}{n}&=
	\frac{n-1}{n}+\frac{\frac{m-1}{m}+\frac{f_{+}(t)}{m}}{n}\\
	&=\frac{mn-m+m-1}{mn}+\frac{f_{+}(t)}{mn}\\
	&=\frac{r-1}{r}+\frac{f_{+}(t)}{r},
	\end{aligned}
	\end{equation*}
	where $r=mn$. Hence $h(t)\in D(\Ff,[a,c])\subseteq \Dd(\Ff)$.
	
	\medskip
	
	\noindent (3)  It suffices to show that for any $h(t)\in \Dd(\Ff)$,  $\Dd^{-1}(\Dd(\Ff),\{h(t)\})$ is a finite set.

	\medskip
	
	Take any $g_1(t)\in\Dd^{-1}(\Dd(\Ff),\{h(t)\})$. We may assume that $g_1(t)\neq 0,1$. There are $I\geq 1$, $c_i\in(a,b]$, $m_i\in\Nn$, $v_i(t)\in (\Ff_{+},[a,c_i])$, and $0\neq g_i(t)=\frac{m_i-1+v_i(t)}{m_i}\in D(\Ff,[a,c_i])$ for $1\leq i\leq I$ such that
	$h(t)=\frac{n-1}{n}+\frac{g(t)}{n}$, where
	$n\in\Nn$ and $g(t)=\sum_{i}g_i(t)$. Let $c\coloneqq\min\{c_i\}$.
	If there are more than one index $i$ such that $m_i\ge 2$, then by the argument of (1), $v_1(t)=0$ and $m_1\in\{1,2\}$, thus $g_1(t)$ belongs to a finite set and we are done. 
	Otherwise, there exists at most one index $i$, such that $m_i\ge 2$. Let $m=\max\{m_i\}$. If $m\ge 2$, then let $i_0$ be the index such that $m=m_{i_0}$, otherwise let $i_0=1$. Now there exists a linear function $f_{+}(t)\in (\Ff_+,[a,c])$ such that $g(t)=\frac{m-1}{m}+\frac{f_{+}(t)}{m}$. Hence $h(t)=\frac{mn-1}{mn}+\frac{f_{+}(t)}{mn}$. 
	
	\medskip
	
	If $h(t)=1$, then $g(t)=\sum_{i} \frac{m_i-1+v_i(t)}{m_i}=1$ and $I\geq 2$ as $g_1(t)\neq 1$. We know that $v_{i_0}(t)+\sum_{i\neq i_0}mv_i(t)= 1$.  Since $\Ff|_a\cup\Ff|_b$ is a DCC set, there exists $\epsilon>0$ such that for any $0\neq f(t)\in\Ff$, $\max\{f(a),f(b)\}>\epsilon$. Thus $m<\frac{1}{\epsilon}$. If $h(t)\neq 1$, then $h(c)<1$. Thus $\frac{mn-1}{mn}\le h(c)<1$ hence $m, n  \leq mn<\frac{1}{1-h(c)}$. Thus in both cases, $m_i$ and the function $g(t)=\sum_{i} \frac{m_i-1+v_i(t)}{m_i}$belong to a finite set. 
	
	\medskip
	
	Since $\mathcal{F}|_{a}\cup\mathcal{F}|_{b}$ is a DCC set, $v_i(a)$ and $v_i(b)$ belong to a finite set. We conclude that $v_1(t)$ belongs to a finite set since it is a linear function. Hence $g_1(t)$ belongs to a finite set, and we are done.
\end{proof}

\begin{remark}
Suppose that $\Ff\ni 1$ is a set of real linear functions, such that for any $f(t)\in \Ff$, $f(t)\geq 0$ for $t\in[a,b]$, and $\Ff|_a\cup\Ff|_b$ is a DCC set. Applying \cite[Proposition 3.4.1]{HMX14} to $\Ff|_a$ and $\Ff|_b$, we know that $D(\Ff)|_a$ and $D(\Ff)|_b$ are DCC sets. If $\Ff$ is a set of constant functions, then $\mathcal{D}(\Ff)|_{b}=\mathcal{D}(\Ff)|_{a}=D(\Ff)|_a$ is a DCC set, too. However, $\mathcal{D}(\Ff)|_{b}$ is no longer a DCC set in general. For example, let $[a,b]=[0,1]$, $\Ff=\{1\}\cup\{nt\}_{n\in\Zz_{> 0}}$, then $\{\frac{2m+1}{m}\}_{m\in\Zz_{> 0}}\subseteq\mathcal{D}(\Ff)|_{1}$ is not a DCC set.    
\end{remark}

\subsection{Divisorially log terminal modification}
For an lc pair $(X,\Delta)$, there exists a dlt modification of $(X,\Delta)$ (see \cite[Proposition 3.3.1]{HMX14}). We generalize that to the setting of linear functional divisors. 

\begin{proposition}[Common dlt modification]\label{prop: common dlt}
Let $a<b$ be two real numbers. Suppose that $(X,\Delta(t))$ is lc for $t\in[a,b]$, where $\Delta(t)$ is a linear functional divisor. Suppose that $S$ is a component of $\Supp\Delta(t)$. Then there is a projective birational morphism $\phi: Y\to X$ satisfying the following.
\begin{enumerate}
\item $Y$ is $\Qq$-factorial,
\item 
\[
K_{Y}+\tilde\Delta(t)+F=\phi^{*}(K_X+\Delta(t)),
\] 
where $F$ is the sum of $\phi$-exceptional divisors, 
\item $(Y, \tilde\Delta(t)+F)$ is dlt for any $t\in(a,b)$, and
\item if $\phi^{-1}(S)\neq\tilde{S}$, then there exists an exceptional divisor $E$, such that $\tilde{S}$ intersects the general fiber of $E\to \phi(E)^{\nu}$, where $\tilde S$ is the strict transform of $S$, and $\phi(E)^{\nu}$ is the normalization of $\phi(E)$.
\end{enumerate} 
\end{proposition}

\begin{proof}
For a fixed interior point $c\in (a,b)$, by \cite[Proposition 3.3.1]{HMX14}, there is a dlt modification of $(X,\Delta(c))$, $\phi: Y \to X$, such that
\[K_{Y}+\tilde\Delta(t)+\sum a_i(t)F_i=\phi^{*}(K_X+\Delta(t)),\]
$(Y, \tilde\Delta(c)+\sum F_i)$ is $\Qq$-factorial dlt, $a_i(t)$ are linear functions, and $a_i(c)=1$, where $\tilde\Delta(t)$ is the strict transform of $\Delta(t)$ and $F_i$ are exceptional divisors. Moreover, there is a divisor of the form $-\tilde S-G$ which is nef over $X$, where $G\geq 0$ is a sum of exceptional divisors whose centers are contained in $S$.

\medskip

We claim that $(Y, \tilde\Delta(t)+\sum a_i(t)F_i)$ is a dlt modification of $(X,\Delta(t))$ for any $t\in(a,b)$. By the property of dlt singularities, there exists a log resolution $Y' \to Y$ of $(Y, \tilde\Delta(c)+\sum a_i(c)F_i)$ such that $a(F',Y,\tilde\Delta(c)+\sum a_i(c)F_i)>-1$ for every exceptional divisor $F'\subseteq Y'$. $Y' \to Y$ is also a log resolution for $(Y, \tilde\Delta(t)+\sum a_i(t)F_i)$ since $c$ is an interior point. Since $a(F',Y,\tilde\Delta(t)+\sum a_i(t)F_i)$ is a linear function, and greater than or equal to $-1$ for $t\in[a,b]$,  $a(F',Y,\tilde\Delta(t)+\sum a_i(t)F_i)>-1$ for $t\in(a,b)$. Thus $(Y, \tilde\Delta(t)+\sum a_i(t)F_i)$ is dlt for $t\in(a,b)$. Since $a_i(c)=1$ and $a_i(t)\le 1$ for $t\in[a,b]$, $a_i(t)\equiv 1$.

\medskip

For (4), let $\mathcal{E}$ be the subset of exceptional divisors 
$$\mathcal{E}\coloneqq\{E'\mid\Supp E'\subseteq \Supp G, E'\cap \tilde{S}\neq \emptyset\}.$$ For any closed point $p \in S$, $\tilde S \cap \phi^{-1}(p) \neq \emptyset$. Since $-\tilde S-G$ is nef over $X$, $\phi^{-1}(p) \subseteq \Supp (\tilde S+G)$ by the negativity lemma (see \cite[Lemma 3.39(2)]{KM98}). This shows $\phi^{-1}(S) = \Supp (\tilde S+G)$. Let $E''\in\mathcal{E}$ such that $\dim \phi(E'')=\dim\cup_{E'\in \mathcal{E}}\phi(E')$. Since $\phi^{-1}(S)=\Supp(\tilde{S}+G)$, 
\[
\phi^{-1}(x)\cap \tilde{S}\cap\Supp(\cup_{E' \in \mathcal{E}}E')\neq \emptyset
\] for any $x\in \phi(E'')$. Thus there exists $E\in\mathcal{E}$, such that $\phi^{-1}(x)\cap \tilde{S}\cap E\neq\emptyset$ for any $x\in \phi(E'')$. 
We conclude that $\phi(E)=\phi(E')$ and $\tilde{S}$ intersects the general fiber of $E\to\phi(E)^{\nu}$ as $\dim \phi({E}'')\ge\dim\phi(E)$. 
\end{proof}

\begin{remark}Proposition \ref{prop: common dlt}(4) will be used in the following cases. 
\begin{enumerate}
    \item If $V\subseteq S$ is a common lc center of $(X,\Delta(t))$ for $t\in[a,b]$, and the coefficient of $S$ in $\Delta(t)$ is not 1, then $\phi^{-1}(S)\neq \tilde{S}$. Hence there is an exceptional divisor $E$, such that $\tilde{S}$ intersects the general fiber of $E\to \phi(E)^{\nu}$.
    \item Suppose that $B$ is another component of $\Delta(t)$, such that $B$ intersects $S$. If $\tilde{B}$ (the strict transform of $B$) does not intersect $\tilde{S}$, then $\phi^{-1}(S)\neq \tilde{S}$. Hence there is an exceptional divisor $E$, such that $\tilde{S}$ intersects the general fiber of $E\to \phi(E)^{\nu}$.
\end{enumerate}
\end{remark}

\subsection{Bounded family}

\begin{definition}[Bounded family]\label{def: bounded family}
A set of varieties $\mathcal{X}$ forms a \emph{bounded family} if there is a projective morphism of varieties $f: Z\to T$, with $T$ of finite type, such that for every $X\in\mathcal{X}$, there is a closed point $t\in T$ and an isomorphism $Z_t \simeq X$, where $Z_t$ is the fiber of $f$ at $t$.
\end{definition}

\begin{definition}[Fano, $\epsilon$-lc Fano]
A log pair $(X,\Delta)$ is called \emph{Fano} if it is lc, $X$ is projective and $-(K_X+\Delta)$ is ample. For a real number $\epsilon\ge0$, $(X,\Delta)$ is called \emph{$\epsilon$-lc Fano} if it is Fano and $\epsilon$-lc.
\end{definition}

We need the following result on boundedness which is known as  Birkar-Borisov-Alexeev-Borisov Theorem.

\begin{theorem}[{\cite[Theorem 1.1]{Bir16b}}]\label{thm: BAB}
Let $d$ be a positive integer and $\epsilon$ a positive real number. Then the set of $\epsilon$-lc Fano varieties $X$ with $\dim X=d$ forms a bounded family.
\end{theorem}

\subsection{Two ray game}\label{subsec: two ray game}
The following variant of the two ray game (see \cite[\S 9]{MP04}) for linear functional divisors will be used in the proof.

\medskip

Let $Y$ be a $\Qq$-factorial normal projective variety with Picard number two. Then the closed cone of effective curves of $Y$ is spanned by two rays $R$ and $S$. Now suppose that there is a sequence of flips $g: Y\dashrightarrow Y'$ with respect to a fixed pseudo-effective divisor. Then the cone of effective curves of $Y'$ is also spanned by two rays $R'$ and $S'$. Possibly switching the roles of $R$ and $S$, we may assume that $R$ is flipped in $Y$ and $S'$ is spanned by the last flipping curve in $Y'$.

\medskip

Let $a<b$ be two real numbers. Suppose that $(Y,\Delta(t))$ is a log pair, such that $\Delta(t)$ is a linear functional divisor of $t$ on $[a,b]$ and $K_{Y}+\Delta(t)$ is numerically trivial for $t\in[a,b]$. Suppose that both rays $S$ and $R'$ are contractible, which induce divisorial contractions. Let $f:Y\to X$ be the contraction of $S$, and let $f':Y'\to X'$ be the contraction of $R'$. Let $E_1$ and $E_2'$ be the corresponding exceptional divisors respectively. Let $E_1'$, $\Gamma'(t)$ be the strict transforms of $E_1,\Gamma(t)$ on $Y'$ respectively and $E_2$ the strict transform of $E_2'$ on $Y$.

\medskip

Write
\[
K_{Y}+\Delta(t)\coloneqq K_{Y}+\Gamma(t)+u(t)E_1+v(t)E_2,
\]
where $u(t), v(t)$ are linear functions, $E_1\nsubseteq\Supp\Gamma(t)$ and $E_2\nsubseteq\Supp\Gamma(t)$.

\begin{proposition}[Two ray game]\label{prop: two ray game}
	Under the above assumptions, suppose that $(Y,\Delta(t))$ is klt for $t\in(a,b)$, $K_{Y}+\Delta(t)\equiv 0$, and $u(a)=v(a)=1$. Then there exist linear functions $u'(t),v'(t)$ and $\epsilon>0$, such that $u(t)<u'(t)\le1$ and $v(t)<v'(t)\le 1$ for any $t\in(a, +\infty)$, and one of the following two cases holds.
	\begin{enumerate}
		\item $E_2$ intersects $S$, $v'(t)<1$ for any $t\in(a, +\infty)$, such that $Y\to X$ is $(K_Y+\Gamma(t)+u'(t)E_1+v'(t)E_2)$-trivial, $(Y,\Gamma(t)+u'(t)E_1+v'(t)E_2)$ is lc for $t\in(a,a+\epsilon)$, and there exists a common lc center of the pair contained in $E_1\cup E_2$.		
		\item $E_1'$ intersects $R'$, $u'(t)<1$ for any $t\in(a, +\infty)$, such that $Y'\to X'$ is $(K_{Y'}+\Gamma'(t)+u'(t)E_1'+v'(t)E_2')$-trivial, $({Y'}+\Gamma'(t)+u'(t)E_1'+v'(t)E_2')$ is lc for $t\in(a, a+\epsilon)$, and there exists a common lc center of the pair contained in $E_1'\cup E_2'$.
	\end{enumerate}
\end{proposition}
\begin{proof}
Since both $u(t)$ and $v(t)$ are linear functions, $u(a)=v(a)=1$, and $(Y,\Delta(t))$ is klt for $t\in(a,b)$. We have $u(t)<1$ and $v(t)<1$ for any $t\in(a,+\infty)$. Let $D(t)=K_Y+\Gamma(t)+E_1+E_2$. We will consider the following two cases.

\medskip 

\noindent {\bf Case 1.}	There exists $\epsilon>0$, such that $D(t)\cdot S>0$ when $t\in(a,a+\epsilon)$. Since $E_1 \cdot S<0$, it follows that 
$$E_2 \equiv \frac{1}{1-v(t)}(D(t) - (1-u(t))E_1)$$ 
is $S$-positive when $t\in(a,a+\epsilon)$, and $E_2$ intersects $S$. Let $-E_1\cdot S=\lambda (E_2 \cdot S)$ for some $\lambda>0$, and
	\[
	\zeta(t)\coloneqq \lct(Y, \Gamma(t)+u(t)E_1+v(t)E_2; E_1+\lambda E_2).
	\] 
		
	Then 
	$$(K_Y+\Gamma(t)+(u(t)+\zeta(t))E_1+(v(t)+\lambda\zeta(t))E_2)\cdot S=0.$$ 
	
		Since $(Y,\Delta(t))$ is klt for $t\in(a,b)$, $\zeta(t)>0$ for $t\in(a,a+\epsilon)$. By Lemma \ref{le: linearity of lct}, possibly replacing $\epsilon$ with a smaller positive real number, we may assume that $\zeta(t)$ is linear for $t\in(a, a+\epsilon)$, and there is a common lc center of $(Y,\Gamma(t)+(u(t)+\zeta(t))E_1+(v(t)+\lambda\zeta(t))E_2)$ contained in $E_1\cup E_2$ for any $t\in(a,a+\epsilon)$. 
	
	\medskip
	
	For $t\in (a, a+\epsilon)$, define
	\[
	u'(t)=u(t)+\zeta(t),\quad v'(t)=v(t)+\lambda\zeta(t),
	\]
	and extend the domain of $u'(t)$ and $v'(t)$ to $(-\infty,+\infty)$ by linearity.
	
	\medskip
	
	Since $\zeta(t)>0$ for $t\in(a,a+\epsilon)$, $u(a)=v(a)=1$ implies that $\zeta(a)=0$. Thus $u(t)<u'(t)\le1$, $v(t)<v'(t)\le1$ when $t\in(a, a+\epsilon)$. Suppose that $v'(t)\equiv 1$ for $t\in(a,a+\epsilon)$, then 
	\begin{equation*}
	\begin{aligned}
	0=&(K_Y+\Gamma(t)+u'(t)E_1+v'(t)E_2)\cdot S\\
	=&(K_Y+\Gamma(t)+u'(t)E_1+E_2)\cdot S\\
	\ge&(K_Y+\Gamma(t)+E_1+E_2)\cdot S>0,
	\end{aligned}
	\end{equation*}
	a contradiction. Thus $v'(t)<1$ for any  $t\in(a,a+\epsilon)$. Since $u'(a)=1,v'(a)=1$, and $u'(t)$, $v'(t)$ are linear functions, we have $u(t)<u'(t)\le1$, $v(t)<v'(t)<1$ for any $t\in(a, +\infty)$. 
	
	\medskip
	
	\noindent {\bf Case 2.} There exists $\epsilon>0$, such that $D(t)\cdot S\le0$ when $t\in(a,a+\epsilon)$. Since $D(t)\equiv (1-u(t))E + (1-v(t))E'\neq0$ is effective when $t>a$, $D(t)\cdot R>0$. We claim that $D'(t) \cdot R' >0$ for any $t>a$, where $D'(t)$ is the strict transform of $D(t)$ on $Y'$. Since $Y \dashrightarrow Y'$ is a sequence of flips, by the induction on the number of flips, it suffices to consider the case when $Y\dashrightarrow Y'$ is a flip. Since $D(t) \cdot R>0$ and $D'(t)\cdot S'<0$, $D'(t) \cdot R' >0$ and the claim is proved. Now (2) follows from the same argument as in Case 1 on $Y'$ for $R'$. 
\end{proof}

\section{Proof of Theorem \ref{thm: linear global ACC}}\label{sec: proof}

In this section, we will prove Theorem \ref{thm: linear global ACC}. It is a consequence of Theorem \hyperlink{thm: N}{N}. The proof of Theorem \hyperlink{thm: N}{N} and Theorem \hyperlink{thm: P}{P} proceeds by induction on dimensions. The ideas of the proof of Theorem \ref{thm: linear global ACC} go back to Shokurov and \cite{MP04}.

\medskip

\noindent {\bf \hypertarget{thm: N}{Theorem N}} (ACC for numerically trivial pairs){\bf.} 
{\it 	Let $n \in \Nn$, and let $a<b$ be two real numbers. Suppose that $\Ff\ni 1$ is a set of real linear functions, such that for any $f(t)\in \Ff$, $f(t)\geq 0$ for $t\in[a,b]$, and $\Ff|_a\cup\Ff|_b$ is a DCC set. Then there is a finite subset $D_0 \subseteq \Dd(\Ff)$ satisfying the following. If
	\begin{enumerate}
		\item $X$ is a normal projective variety of dimension $n$,
	\item $\Delta(t) \in \Dd(\Ff)$ is a linear functional divisor of $t$ on $[a,b]$,
		\item there exists $a< b_X \leq b$, such that $(X,\Delta(t))$ is lc for any $t\in[a, b_X]$, and
		\item $K_{X}+\Delta(t)$ is numerically trivial for any $t\in[a, b]$,
	\end{enumerate}
	then $\Delta(t) \in D_0$. 
	}

\begin{remark}
	Comparing with Theorem \ref{thm: linear global ACC}, we require $\Delta(t)\in\Dd(\Ff)$ in Theorem \hyperlink{thm: N}{N}. This facilitates the induction argument since by Lemma \ref{le: properties of derived set}, $\Dd(\Dd(\Ff)) = \Dd(\Ff)$.
\end{remark}

\noindent {\bf \hypertarget{thm: P}{Theorem P}} (ACC for numerically trivial pairs of Picard number $1$){\bf .} {\it 	Let $n \in \Nn$, and let $a<b$ be two real numbers. Suppose that $\Ff\ni 1$ is a set of real linear functions, such that for any $f(t)\in \Ff$, $f(t)\geq 0$ for $t\in[a,b]$, and $\Ff|_a\cup\Ff|_b$ is a DCC set. Then there is a finite subset $D_0 \subseteq \Dd(\Ff)$ satisfying the following. If
	\begin{enumerate}
		\item $X$ is a normal projective $\Qq$-factorial Fano variety of dimension $n$ with Picard number $1$,
		\item $\Delta(t) \in \Dd(\Ff)$ is a linear functional divisor of $t$ on $[a,b]$,
		\item there exists $a< b_X \leq b$, such that $(X,\Delta(t))$ is dlt for any $t\in(a, b_X]$, and
		\item $K_{X}+\Delta(t)$ is numerically trivial for any $t\in[a, b]$,
	\end{enumerate}
	then $\Delta(t) \in D_0$. 
}

\begin{remark}\label{rmk: after thm P}
	Comparing with Theorem \hyperlink{thm: P}{N}, Theorem \hyperlink{thm: P}{P} adds ``$\Qq$-factorial Fano variety with Picard number $1$'' in condition (1), and requires $(X,\Delta(t))$ to be dlt for $t\in(a, b_X]$ in condition (3). Note that $(X,\Delta(t))$ is automatically lc for $t\in[a, b_X]$ since lc is a closed condition.
\end{remark}

In the following inductive argument for Theorem \hyperlink{thm: N}{N} and Theorem \hyperlink{thm: P}{P}, Theorem \hyperlink{thm: N}{N\textsubscript{$n$}}, Theorem \hyperlink{thm: N}{N\textsubscript{$\leq n$}}, etc., stand for Theorem \hyperlink{thm: N}{N} for varieties of dimension $n$, $\leq n$, etc. 

\begin{proposition}\label{prop: P+N implies N}
	Theorem \hyperlink{thm: P}{P\textsubscript{$n$}} and Theorem \hyperlink{thm: N}{N\textsubscript{$\leq n-1$}} imply Theorem \hyperlink{thm: N}{N\textsubscript{$n$}}.
\end{proposition}
\begin{proof}
	Suppose that $(X, \Delta(t))$ satisfies the conditions of Theorem \hyperlink{thm: N}{N\textsubscript{$n$}}. By Proposition \ref{prop: common dlt}, there exists a common dlt modification $\phi: Y\to X$ for $t\in(a,b_X)$, such that
	\begin{equation*}
	K_{Y}+F+\tilde\Delta(t)=\phi^{*}(K_{X}+\Delta(t)),
	\end{equation*}
	where $F$ is a reduced exceptional divisor and $\tilde\Delta(t)$ is the strict transform of $\Delta(t)$. We may assume $\tilde\Delta(t)\neq0$. Let $w(t)\Delta$ be a summand of $\Delta(t)$ with $w(t) \in \Dd(\Ff)$, and let $w(t)\tilde \Delta$ be its strict transform on $Y$.

	\medskip
	
	Since $K_{X}+\Delta(t)$ is numerically trivial, 
	\begin{equation*}
	K_{Y}+F+\tilde\Delta(t) - w(t)\tilde\Delta \equiv -w(t)\tilde\Delta
	\end{equation*}
	is not pseudo-effective for $t\in(a,b_X)$. Let $Y_1:=Y\dashrightarrow Y_2\dashrightarrow Y_3\dashrightarrow\cdots$ be the sequence of a $(K_Y+F+\tilde\Delta(t)-w(t)\tilde\Delta)$-MMP for some $t\in (a,b_X)$, then it is $\tilde\Delta$-positive. 
	\medskip
	
	Suppose that some component $S$ of $F$ is contracted in some step $Y_k\dashrightarrow Y_{k+1}$ of the MMP. Let $F_k$, $S_k$, $\tilde\Delta_k(t)$ and $\tilde\Delta_k$ be the strict transforms of $F$, $S$, $\tilde\Delta(t)$ and $\tilde{\Delta}$ on $Y_k$ respectively. Then $S_k$ intersects $\tilde\Delta_k$. By the adjucntion formula (Proposition \ref{prop: adjunction}) and Lemma \ref{le: properties of derived set}, 
	\[
	(K_{Y_k}+F_k+\tilde\Delta_k(t))|_{S_k}=K_{S_k}+\Theta_k(t),
	\]
	where $(S_k,\Theta_k(t))$ is dlt for $t\in(a,b_X)$, and $\Theta_k(t)\in\Dd(\Dd(\Ff))=\Dd(\Ff)$. Since $\dim S_k<\dim X$, by Theorem \hyperlink{thm: N}{N\textsubscript{$\leq n-1$}}, there exists a finite $D_1\subseteq \Dd(\Ff)$, such that $\Theta_k(t)\in D_1$. By Lemma \ref{le: properties of derived set}(3), $w(t)$ belongs to the finite set $\Dd^{-1}(\Dd(\Ff),D_1)$. Hence we may assume that no component of $F$ is contracted in the MMP. 
	
		\medskip
	
	We may run a $(K_Y+F+\tilde\Delta(c)-w(c)\tilde\Delta)$-MMP with scaling of an ample divisor for any $c\in(a,b_X)$, and reach a Mori fiber space (for a reference, see \cite[Theorem 2.15]{Loh13}), $\pi: Y' \to Z$. Since each step of the MMP is $\tilde \Delta$-positive, the final Mori fiber space can be assumed to be the same for all $c$. Moreover, $\tilde \Delta'$ intersects any general fiber of $\pi$, where $\tilde\Delta'$ is the strict transform of $\tilde\Delta$ on $Y'$. Now if $Z$ is not a point, then for the general fiber ${Y'_z}$ of $\pi$, we have
	\begin{equation*}
	K_{Y'_z}+F'|_{Y'_z}+{\tilde\Delta'(t)}|_{Y'_z} \equiv 0,
	\end{equation*}
	where $F'$ and $\tilde\Delta'(t)$ are the strict transforms of $F$ and $\tilde{\Delta(t)}$ on $Y'$ respectively. 
	Moreover, ${\tilde\Delta'(t)}|_{Y'_z}\in \Dd(\Ff)$, and $w(t)\tilde\Delta'|_{Y'_z}$ is nonzero. By Theorem \hyperlink{thm: N}{N\textsubscript{$\leq n-1$}}, $w(t)\in D_1$. If $Z$ is a point, then $Y'$ has Picard number one. If $F'\neq 0$, then $F'$ intersects $\tilde{\Delta}'$. We are done again by the adjucntion formula (Proposition \ref{prop: adjunction}), Lemma \ref{le: properties of derived set}, and Theorem \hyperlink{thm: N}{N\textsubscript{$\leq n-1$}}.
	If $F'=0$, then $F=0$ and $(Y, F+\tilde\Delta(t))$ is klt. Thus $(Y', F'+\tilde\Delta'(t))$ is also klt as $K_Y+ F+\tilde\Delta(t)\equiv 0$. The proposition follows from Theorem \hyperlink{thm: P}{P\textsubscript{$n$}}.
\end{proof}

\begin{proposition}\label{prop: bounded implies thm P}
	Let $\epsilon>0$ be a fixed real number. If we further assume that $X$ is $\epsilon$-lc in Theorem \hyperlink{thm: P}{P\textsubscript{$n$}}, then Theorem \hyperlink{thm: P}{P\textsubscript{$n$}} holds. In particular, both Theorem \hyperlink{thm: N}{N\textsubscript{$1$}} and Theorem \hyperlink{thm: P}{P\textsubscript{$1$}} hold.
\end{proposition}
\begin{proof}
	By Theorem \ref{thm: BAB}, all the $X$ satisfying the conditions of Proposition \ref{prop: bounded implies thm P} belong to a bounded family. 
	
	\medskip
	
By \cite[Lemma 2.24]{Bir16a}, there exists a positive integer $r$ such that $rK_{X}$ is Cartier for each $X$. There exists a general curve $C_X$ in the smooth locus of $X$, such that $rK_X \cdot C_X\in \Zz$ belong to a finite set. 

\medskip

	It is enough to prove the proposition for any sequence $\{(X_i, \Delta_i(t))\}_{i\in\Nn}$. Let $\Delta_i(t) = \sum_{j=1}^{s_i} w_{i,j}(t)\Delta_{i,j}$ where
	\begin{equation}\label{eq: w}
	w_{i,j}(t) = \frac{m_{i,j} -1 + v_{i,j}(t)}{m_{i,j}}, \quad v_{i,j}(t) \in \Ff_+.
	\end{equation}
	Note that $\Delta_{i,j} \cdot C
	_{X_i} = r_{i,j}$ is a positive integer, and
	\begin{equation*}
	0 = K_{X_i} \cdot C_{X_i} + \sum_{j=1}^{s_i} w_{i,j}(t) (\Delta_{i,j} \cdot C_{X_i}).
	\end{equation*}
	Since there are finitely many possibilities for $K_{X_i} \cdot C_{X_i}$, passing to a subsequence, we may assume that $-K_{X_i} \cdot C_{X_i}=K$ for any $i$. Thus
	\begin{equation}\label{eq: intersection with C}
	K = \sum_{j=1}^{s_i} r_{i,j} w_{i,j}(t), \quad r_{i,j} \in \Zz_{> 0}.
	\end{equation}
	
	Let $\tau = \min\{\Ff|_{a}\cup \Ff|_{b}\backslash\{0\}\}$, $c = \frac{a+b}{2}$. By linearity, $w_{i,j}(c)\geq \min\{\frac{1}{2},\frac \tau 2\}$. Comparing with \eqref{eq: intersection with C}, we see that $s_i$ and $r_{i,j}$ are all bounded. Hence passing to a subsequence, we may assume $s_i = s$ and $r_{i,j}=r_j$ for all $i$. However, $m_{i,j}$ in \eqref{eq: w} may not be bounded. Passing to a subsequence and possibly switching the order, we may assume that there exists $0 \leq s' \leq s$, such that $m_{i,1},\ldots,m_{i,s'}$ are bounded, and $m_{i,s'+1},\ldots,m_{i,s}$ are unbounded. We may further assume that $m_{i,j} = m_j$ for each $j\le s'$, and $\{m_{i,j}\}_{i=1}^{\infty}$ is strictly increasing for each $s\ge j>s'$. 
	
	\medskip
	
	If $\{v_{i,j}(t)\}$ is not a finite set, then passing to a subsequence, we may assume that
	\begin{equation}\label{eq: w are different}
	\bigcup_{i=1}^l\{v_{i,j}(t) \mid j=1,2,\ldots,s\}\subsetneq \bigcup_{i=1}^{l+1}\{v_{i,j}(t) \mid j=1,2,\ldots,s\}
	\end{equation} for each $l$.
	
	\medskip
	
	Since $v_{i,j}(a),v_{i,j}(b)$ belong to the DCC set $\sum \Ff|_a\cup \sum\Ff|_b$, passing to a subsequence, we may assume that for fixed $j$, both $\{v_{i,j}(a)\}_{i\in \Nn}$ and $\{v_{i,j}(b)\}_{i\in \Nn}$ are non-decreasing sequences. Thus, $v_{i_2,j}(t)\ge v_{i_1,j}(t)$ for $t\in[a,b]$ and any $i_2\ge i_1$, and $w_{i_2,j}(t)\ge w_{i_1,j}(t)$ for $t\in[a,b]$ and any $i_2\ge i_1$, $j\le s'$.
	
	\medskip
	
	For a fixed $j$, if $w_{i,j}(t)\equiv1$ for infinitely many $i$, then passing to a subsequence, we may assume that $w_{i,j}(t)\equiv1$ for any $i$. Subtracting these terms from \eqref{eq: intersection with C}, we may assume that $w_{i,j}(t)$ is not identically equal to $1$ for any $i, j$.
	
	\medskip
	
	When $s =s'$, that is, $m_{i,j}$ are bounded for all $j$, comparing term by term in \eqref{eq: intersection with C}, we have $v_{i_2,j}(t)=v_{i_1, j}(t)$ for any $i_2\ge i_1$. This contradicts \eqref{eq: w are different}.
	
	\medskip
	
	When $s'<s$, split \eqref{eq: intersection with C} into two parts,
	\begin{equation}\label{eq: split}
	\begin{split}
	&P_i(t) \coloneqq \sum_{j = 1}^{s'}  r_{j} w_{i,j}(t) \text{~and~} Q_i(t) \coloneqq \sum_{j = s'+1}^{s} r_{j} w_{i,j}(t).
	\end{split}
	\end{equation}
	
	Since $w_{i,j}(t)$ is not identically equal to $1$, we may choose $c'\in(a, b)$ such that $w_{1,j}(c')<1$ for any $j$.
	Since $\{m_{i,j}\}_{i\in\Nn}$ is unbounded for each $s'+1\leq j \leq s$, there exists $i'$ such that
	\begin{equation*}
	w_{i',j}(c') \ge\frac{m_{i',j} -1}{m_{i',j}}> w_{1,j}(c')
	\end{equation*}
	for any $s'+1\leq j \leq s$.
    
    \medskip
    	
	Thus $Q_{i'}(c')>Q_{1}(c')$. Since $w_{i',j}(t)\ge w_{1,j}(t)$ for $t\in[a,b]$, $w_{i',j}(c')\ge w_{1,j}(c')$, and thus $P_{i'}(c')\ge P_{1}(c')$. This is a contradiction since $P_{i'}(c') + Q_{i'}(c') =P_{1}(c') + Q_{1}(c')=K$. Hence $v_{i,j}(t)$ belongs to a finite set. Moreover, for any $j$, either $m_{i,j}$ is bounded or $w_{i,j}(t)\equiv 1$. Thus $w_{i,j}(t)$ belongs to a finite set $D_0$.

	\medskip
	
	When $n=1$, it suffices to prove Theorem \hyperlink{thm: N}{N\textsubscript{$1$}}. We may assume $\Delta(t)\neq 0$. Since $K_X+\Delta(t) \equiv 0$ and $X$ is normal, $X$ is $\Pp^1$. Hence Theorem \hyperlink{thm: N}{N\textsubscript{$1$}} holds.	
\end{proof}

\begin{proposition}\label{prop:not bounded}
	Theorem \hyperlink{thm: N}{N\textsubscript{$\leq n-1$}} implies Theorem \hyperlink{thm: P}{P\textsubscript{$n$}}.
\end{proposition}
\begin{proof}
	It suffices to prove Theorem \hyperlink{thm: P}{P\textsubscript{$n$}} for any sequence $\{(X_i, \Delta_i(t))\}_{i\in\Nn}$. By Proposition \ref{prop: bounded implies thm P}, we may assume $n\ge2$. Passing to a subsequence, we may assume that the discrepancy of $X_i$ is less than $-1+\frac{1}{2i}$ for any $i$. 
		
	\medskip
	
		\noindent{\bf Case A.} There are infinitely many $i$, such that $\Delta_i(t)$ contains an irreducible component $S_i$ whose coefficient is $1$ for all $t$. By the adjunction formula and Lemma \ref{le: properties of derived set},
	\[
	(K_{X_i}+\Delta_i(t))|_{S_i}=K_{S_i}+\Theta_i(t) {\rm~on~}[a, b_{X_i}],
	\]
	with $\Theta_i(t) \in \Dd(\Ff)$. Since the Picard number of $X_i$ is $1$, $S_i$ intersects any nonzero divisor. By Theorem \hyperlink{thm: N}{N\textsubscript{$\leq n-1$}}, there is a finite set $D'_0 \subseteq \Dd(\Ff)$ such that $\Theta_i(t) \in D'_0$. Thus $\Delta_i(t) \in \Dd^{-1}(\Dd(\Ff), D'_0)$. By Lemma \ref{le: properties of derived set}(3), $\Dd^{-1}(\Dd(\Ff), D'_0)$ is a finite set, and this shows Theorem \hyperlink{thm: P}{P\textsubscript{$n$}}.

	\medskip
	
	\noindent{\bf Case B.} There is no component of $\Delta_i(t)$ whose coefficient is $1$ for all $t$. Since $(X_i, \Delta_i(t))$ is dlt for $t\in(a,b_{X_i}]$, possibly shrinking $(a, b_{X_i}]$, we may assume that it is klt for $t\in(a,b_{X_i}]$. In this case, if the coefficients of $\Delta_i(t), i\in\Zz_{>0}$ do not belong to a finite set, then we have the following claim.
	
	\medskip
	
	\noindent{\bf Claim.} Under the assumptions of Case B, assume that the set of coefficients of $\Delta_i(t), i\in\Zz_{>0}$ is not finite. For any $m \in \Zz_{\geq 0}$, possibly replacing $\{(X_i, \Delta_i(t))\}_{i\in\Zz_{>0}}$ and $\Ff$, we may assume that the following property holds for any $i$.
	\begin{enumerate}
	\item $(X_i, \Delta_i(t))$ is klt for $t\in (a,b_{X_i}]$,
	\item $X_i$ is $\Qq$-factorial with $\rho(X_i)=1$,
	\item the discrepancy of $X_i$ is less than $-1+\frac{1}{2i}$,
	\item $K_{X_i}+\Delta_i(t)\equiv 0$ for $t\in [a,b_{X_i}]$,
	\item  there is a set $\mathfrak{D}_i$ consisting of prime divisors on $X_i$ such that $|\mathfrak{D}_i|=m$, and 
	\item for any $A_i\in \mathfrak{D}_i$, $a(A_i,X_i,\Delta_i(a))=-1$. 
	\end{enumerate}
	
	Assuming this claim, we finish the proof. Choose $m=n+2$, then $(X_i, \Delta_i(a))$ is $\Qq$-factorial lc with $\rho(X_i)=1$, $K_{X_i}+\Delta_i(a)\equiv 0$, and there are $n+1$ coefficients of $\Delta_i(a)$ equal to 1 for any $i$. This contradicts \cite[18.24]{Fli92}, see also \cite[Proposition 5.1]{BZ16}.
\end{proof}
	
\begin{proof}[Proof of the Claim] We prove the claim by induction on $m$. Suppose that the claim holds for $m=m_0\ge0$, and the discrepancy of $X_i$ is less than $-1+\frac{1}{2i}$. For each $i$, by \cite[Corollary 1.4.3]{BCHM10}, there exists a birational morphism $\phi_i: Y_i\to X_i$, such that 
	\[
	K_{Y_i}+a_iA_i=\phi_i^{*}K_{X_i},
	\]
	where $Y_i$ is $\Qq$-factorial, $A_i$ is the only exceptional divisor of $\phi_i$, and $a_i=-a(A_i, X)>1-\frac{1}{2i}>\frac{1}{2}$. Passing to a subsequence, we may assume that $\{a_i\}_{i\in\Nn}$ is an increasing sequence. Let
	\begin{equation*}
	K_{Y_i}+a_i(t)A_i+\tilde\Delta_i(t)=\phi_i^{*}(K_{X_i}+\Delta_i(t)),
	\end{equation*}
	where $a_i(t)$ is a linear function and $\tilde\Delta_i(t)$ is the strict transform of $\Delta_i(t)$ on $Y_i$. Since $\Delta_i(t)$ is effective for $t\in[a, b]$, $a_i(t)\ge a_i$ for $t\in[a, b]$. Note that $a_i(t)<1$ for $t\in(a, b_{X_i}]$ as $(X_i, \Delta_i(t))$ is klt.
	
	\medskip
	
Now we discuss what happens at the end point $t=a$. If $a_i(a)<1$ for infinitely many $i$, then passing to a subsequence, we may assume that $a_i(a)<1$ for any $i$. Moreover, for any $a_i(a)<1$, there exsits $j$, such that $a_j>a_i(a)$, and thus $a_j(a)\ge a_j>a_i(a)$. Passing to a subsequence, we may further assume that $\{a_i(a)\}_{i\in \Nn}$ is strictly increasing. Applying \cite[Proposition 3.4.1]{HMX14} to the set $\Ff|_a$, we know that $\mathcal{D}(\Ff)|_{a}=D(\Ff)|_a$ is a DCC set. Thus the coefficients of $a_i(a)A_i +\tilde\Delta_i(a)$ belong to a DCC set. Since $K_{X_i}+ a_i(a) A_i + \tilde\Delta_i(a) \equiv 0$, by \cite[Theorem 1.5]{HMX14}, their coefficients belong to a finite set. This contradicts to the assumption that $\{a_i(a)\}_{i\in\Nn}$ is strictly increasing.

\medskip

	So we may assume $a_i(a)=1$ for any $i$. Since $a_i(t)<1$ for $t\in(a, b_{X_i}]$, $a_i(t)$ is a decreasing linear function. Recall that $a_i(t) \geq a_i$ for $t\in[a,b]$, there exists $j$ such that $a_j(b)\geq a_j > a_i(b)$ as $a_i(b)<1$. Thus, passing to a subsequence, we may assume that $a_{i_1}(t)<a_{i_2}(t)$ for any $i_1<i_2$ for $t\in(a,b]$.
	
	\medskip
	
    Since
	\begin{equation*}
	K_{Y_i}+(a_i(t)-\frac{1}{2})A_i+\tilde\Delta_i(t) \equiv -\frac{1}{2} A_i
	\end{equation*}
	is not pseudo-effective for $t\in(a,b_{X_i}]$, by \cite[Corollary 1.3.3]{BCHM10}, we may run a $(K_{Y_i}+(a_i(c)-\frac{1}{2})A_i+\tilde\Delta_i(c))$-MMP with scaling of an ample divisor for any $c \in (a,b_{X_i}]$, 
	$$Y_{i,1}:=Y_{i}\dashrightarrow  Y_{i,2}\dashrightarrow\cdots,$$
	and reach a Mori fiber space $\pi_i: Y_i'\to Z_i$. Since each step of the MMP is $A_i$-positive, the final Mori fiber space can be assumed to be the same for any $c$. The general fiber of $\pi_i$ intersects $A_i'$, where $A_i'$ is the strict transform of $A_i$ on $Y_i'$.
	
	\medskip
	
		If there exists a prime divisor $B_i$ on $Y_i$, such that $\phi_i(B_i)\in \mathfrak{D_i}$, and $B_i$ is contracted in some step $Y_{i,k} \dashrightarrow Y_{i,k+1}$ of the MMP, then $Y_i \dashrightarrow Y_{i,k}$ is a sequence of flips since $A_i$ cannot be contracted and the Picard number of $Y_i$ is 2. Let $b_i(t)=a(B_i,X_i,\Delta_i(t))$, $\Delta_i'(t)=\tilde{\Delta_i}(t)-b_i(t)B_i$. Apply Proposition \ref{prop: two ray game} to $Y_i \dashrightarrow Y_{i,k}$. Possibly switching the role of $Y_i$ and $Y_{i,k}$, and shrinking $(a, b_{X_i}]$, we may assume that $B_i$ intersects the general fiber of $A_i\to \phi_i(A_i)^{\nu}$, there exist linear functions $a_i'(t),b_i'(t)$ such that $Y_i\to X_i$ is $(K_{Y_i}+a_i'(t)A_i+b_i'(t)B_i+\Delta_i'(t))$-trivial, 
	$({Y_i},a_i'(t)A_i+b_i'(t)B_i+\Delta_i'(t))$ is lc for any $t\in[a,b_{X_i}]$, and there exists a common lc center contained in $A_i\cup B_i$, where $\phi_i(A_i)^{\nu}$ is the normalization of $\phi_i(A_i)$, $1\ge a_i'(t)\ge a_i(t)$, $1> b_i'(t)>b_i(t)$ for $t\in(a, b_{X_i}]$ and $a_i'(a)=b_i'(a)=1$. Passing to a subsequence, we may assume that for any $i_1<i_2$, $a_{i_1}'(t)\le a_{i_2}'(t)\le 1$, $b_{i_1}'(t)< b_{i_2}'(t)<1$ for $t\in(a, b]$. 
	
	\medskip
	
	Suppose that there are infinitely many $i$ such that $a_i'(t)\equiv 1$. Passing to a subsequence, we may assume that $a_i'(t)\equiv 1$ for any $i$. By Proposition \ref{prop: common dlt}, we may take a common dlt modification $\chi_i: W_i\to Y_i$ of $K_{Y_i}+a_i'(t)A_i+B_i+\Delta_i'(t)$ for $t\in(a, b_{X_i})$. Let $\tilde{A}_i$ and $\tilde{B}_i$ be the strict transforms of $A_i$ and $B_i$ on $W_i$ respectively. Recall that $B_i$ intersects the general fiber of $A_i\to \phi_i(A_i)^{\nu}$, by Proposition \ref{prop: common dlt}(4), either $\chi_i^{-1}B_i=\tilde{B}_i$ or there is an exceptional prime divisor $E_i$, such that $\tilde{B}_i$ intersects the general fiber of $\chi_i|_{E_i}:E_i\to \chi_i(E_i)^{\nu}$. In the former case, $\tilde{B}_i$ intersects the general fiber of $\tilde{A_i}\to \phi(A_i)^{\nu}$. Recall that $W_i\to X_i$ is $\chi_i^{*}(K_{Y_i}+a_i'(t)A_i+B_i+\Delta_i'(t))$-trivial, applying the adjunction formula to $\chi_i^{*}(K_{Y_i}+a_i'(t)A_i+B_i+\Delta_i'(t))$
	on $\tilde{A_i}$ and restricting to the general fiber of $\tilde{A_i}\to \phi_i(A_i)^{\nu}$, we know that $b_{i}'(t)$ belongs to a finite set by Theorem \hyperlink{thm: N}{N\textsubscript{$\leq n-1$}} (applied to the set $\Ff\cup\{b_i'(t)\}_{i\in\Nn}$) and Lemma \ref{le: properties of derived set}(3), a contradiction. In the latter case, applying the adjunction formula to $\chi_i^{*}(K_{Y_i}+A_i+b_i'(t)B_i+\Delta_i'(t))$ on $E_i$ and restricting to the general fiber of $E_i
	\to \chi_i(E_i)^{\nu}$, we know that $b_{i}'(t)$ belongs to a finite set by Theorem \hyperlink{thm: N}{N\textsubscript{$\leq n-1$}} (applied to the set $\Ff\cup\{b_i'(t)\}_{i\in\Nn}$) and Lemma \ref{le: properties of derived set}(3), also a contradiction.
	
	\medskip
	
	Hence, passing to a subsequence, we may assume that $a_i'(t)<1$ for any $i$. Possibly switching the role of $A_i$ and $B_i$, by Proposition \ref{prop: common dlt}, we may take a common dlt modification $\chi_i: W_i\to Y_i$ of $K_{Y_i}+a_i'(t)A_i+b_i'(t)B_i+\Delta_i'(t)$ for $t\in(a, b_{X_i})$, such that there exists an exceptional prime divisor $E_i$, such that the strict transform of $B_i$ intersects the general fiber of $E_i
	\to \chi_i(E_i)^{\nu}$. Applying the adjunction formula to $\chi_i^{*}(K_{Y_i}+a_i'(t)A_i+b_i'(t)B_i+\Delta_i'(t))$ on $E_i$ and restricting to the general fiber of $E_i
	\to \chi_i(E_i)^{\nu}$, we know that $b_{i}'(t)$ belongs to a finite set by Theorem \hyperlink{thm: N}{N\textsubscript{$\leq n-1$}} (applied to the set $\Ff\cup\{a_i'(t)\}_{i\in\Nn}\cup\{b_i'(t)\}_{i\in\Nn}$) and Lemma \ref{le: properties of derived set}(3), a contradiction.
	
	\medskip
	
	Thus we reach a Mori fiber space $\pi_i: Y_i'\to Z_i$ such that $A_i$ and all the prime divisors in $\mathfrak{D}_i$ are not contracted in $Y_i\dashrightarrow Y_i'$.
	
	\medskip
	
	If $\dim(Z_i)>0$, then we get a contradiction by restricting to the general fiber of $\pi_i$, and Theorem \hyperlink{thm: N}{N\textsubscript{$\leq n-1$}} (applied to the set $\Ff\cup\{a_i(t)\}_{i\in\Nn}$).
	
	\medskip
	
	If $Z_i$ is a point, then the Picard number of $Y_i'$ is $1$. Replace $X_i$ by $Y_i'$, $\Delta_i(t)$ by the strict transform of $a_i(t)A_i+\tilde{\Delta_i}(t)$ on $Y_i'$, $\mathfrak{D}_i$ by the set consisting of strict transforms of divisors in the original $\mathfrak{D}_i$ plus $A_i$, and $\Ff$ by $\Ff\cup\{a_i(t)\}_{i\in\Nn}$. Then $|\mathfrak{D}_i|=m_0+1$ and this finishes the induction.
	\end{proof}

Theorem \hyperlink{thm: N}{N} and Theorem \hyperlink{thm: P}{P} follow from the above propositions.

\begin{proof}[Proof of Theorem \hyperlink{thm: N}{N} and Theorem \hyperlink{thm: P}{P}]
	By Proposition \ref{prop: bounded implies thm P}, both Theorem \hyperlink{thm: P}{P\textsubscript{$1$}} and Theorem \hyperlink{thm: N}{N\textsubscript{$1$}} hold. By induction on the dimension, we may assume that Theorem \hyperlink{thm: N}{N\textsubscript{$\leq n-1$}} and Theorem \hyperlink{thm: P}{P\textsubscript{$\leq n-1$}} hold. By Proposition \ref{prop:not bounded}, Theorem \hyperlink{thm: P}{P\textsubscript{$n$}} holds. Then by Proposition \ref{prop: P+N implies N}, Theorem \hyperlink{thm: N}{N\textsubscript{$n$}} holds.
\end{proof}

\begin{proof}[Proof of Theorem \ref{thm: linear global ACC}]
	Since $(X,\Delta(t))$ is log canonical for any $t\in[a,b_X]$, $\Delta(t)\in\Dd(\Ff)\cap\Ff$, Theorem \ref{thm: linear global ACC} follows from Theorem \hyperlink{thm: N}{N}.
\end{proof}

\section{Proofs and remarks}\label{sec: proof of thm and cor}

\subsection{Proofs of theorems and corollaries}
\begin{proof}[Proof of Theorem \ref{thm: ACC for Local LCT-polytopes}]
	Suppose to the contrary that there exist a sequence $\{(X_i, \Delta_i;D_{i,1},\ldots,D_{i,s})\}$, $c_i\in \Rr_{>0}$, and linear functions $t_i(\lambda)$ of $\lambda$ defined on $[0,+\infty)$, such that for any $\lambda\in[0,c_i]$,
	\begin{equation*}\label{eq: t_i}
	t_i(\lambda)=\sup\{t \mid ({X_i}, \Delta_i+\sum_{j=1}^{s-1}(a_{j}\lambda+b_{j})D_{i,j}+tD_{i,s}){\rm~is~lc}\},
	\end{equation*}
	and $t_i(\lambda)<t_{i+1}(\lambda)$ for any $\lambda>0.$ In particular, all $t_i(\lambda)\ge0$ for $t\in[0,c_1]$. Set $c=c_1$.

	\medskip 
	
	Consider linear functional divisors $\Delta_i(\lambda)$ of $t$ on $[0,c]$,
	\begin{equation*}
	\Delta_i(\lambda)\coloneqq \Delta_i+\sum_{j=1}^{s-1}(a_{j}\lambda+b_{j})D_{i,j}+t_i(\lambda)D_{i,s},
	\end{equation*} 
	and let $\Ff$ be the set of coefficients of $\Delta_i(\lambda)$. We claim that both $\Ff|_0$ and $\Ff|_c$ are DCC sets. Indeed,
	\begin{enumerate}
		\item $\Ii$ is DCC,
		\item $\{b_{j}\}_{1 \leq j < s}$ and $\{a_{j}c+b_{j}\}_{1 \leq j < s}$ are finite sets,
		\item $t_i(0)= b_s$ for all $i$, and
		\item $0<t_i(c)<t_{i+1}(c)$.
	\end{enumerate}
	
	By Lemma \ref{le: linearity of lct}, $(X_i, \Delta_i(\lambda))$ has at least one common lc center contained in $\Supp D_{i,s}$ for any $\lambda\in[0, c_i]$. If some component $E_i$ of $D_{i,s}$ is a common lc center for infinitely many $i$, then for $\lambda \in [0,c_i]$,
	\begin{equation*}
	{\rm mult}_{E_i}(\Delta_i)+ \sum_{j=1}^{s-1}(a_{j}\lambda+b_{j}){\rm mult}_{E_i}(D_{i,j})+t_i(\lambda){\rm mult}_{E_i}(D_{i,s})= 1,
	\end{equation*}
	with ${\rm mult}_{E_i}(D_{i,s}) \neq 0$. By linearity, this also holds for $\lambda \in [0,c]$. For a fixed $\lambda\in (0,c)$,
	\begin{equation*}
	{\rm mult}_{E_i}(\Delta_i), \quad \sum_{j=1}^{s-1}(a_{j}\lambda+b_{j}){\rm mult}_{E_i}(D_{i,j}), \quad {\rm mult}_{E_i}(D_{i,s})
	\end{equation*}
	are all in a DCC set, thus $t_i(\lambda)$ belongs to an ACC set, a contradiction. So we may assume that any such lc center is not a component of $D_{i,s}$ for any $i$.
	
	\medskip
	
	By Proposition \ref{prop: common dlt}, there is a common dlt modification
	\[
	K_{Y_i}+\tilde\Delta_i(\lambda)+F_i=\phi_i^{*}(K_{X_i}+\Delta_i(\lambda)),
	\] 
	and there is an exceptional prime divisor $E_i$ such that the strict transform of $D_{i,s}$ intersects the general fiber of $E_i \to \phi_i(E_i)^\nu$. Applying the adjunction formula (Proposition \ref{prop: adjunction}) on $E_i$, and restricting to the general fiber $T_i$ of $E_i \to \phi_i(E_i)^\nu$, we get
	\begin{equation*}\label{eq: adjunction on E}
	\begin{aligned}
	K_{T_i} +\Theta_i(\lambda)\coloneqq&(K_{Y_i}+F_i+\tilde\Delta_i(\lambda))|_{T_i}\\
	=&\phi_i^{*}(K_{X_i}+\Delta_i(\lambda))|_{T_i} \equiv 0.
	\end{aligned}
	\end{equation*} 
	By Theorem \hyperlink{thm: N}{N\textsubscript{$\leq n-1$}} (applied to the set $\mathcal{F}=\{t_i(\lambda)\}_{i\in\Nn}\bigcup\cup_{1\leq s-1}(a_j\lambda+b_j)\Ii$) and Lemma \ref{le: properties of derived set}(3), $t_i(\lambda)$ belongs to a finite set, a contradiction. 
\end{proof}

The following lemmas are used in the proof of Theorem \ref{thm: ACC for LCT-polytopes}. Recall that the dimension of a convex set $T\subseteq \Rr^s$ is defined to be the dimension of its affine hull, which is the intersection of all affine subspaces containing $T$. For a set $S \subseteq \Rr^s$ and a point $\alpha\in \Rr^s$, let 
\begin{equation*}\label{eq: cone}
C_{\alpha}(S)\coloneqq \{\alpha+t(s-\alpha) \mid s \in S, t\in \Rr_{\geq 0}\}
\end{equation*} 
be the cone generated by $S$ with vertex $\alpha$.

\begin{lemma}\label{lem:commona ray}
Let $T_1\supseteq T_2\supseteq \cdots$ be a decreasing sequence of compact polytopes of dimension $s$ in $\Rr^s$. Suppose that $T\subseteq \cap_{i=1}^{\infty} T_i$ is a closed convex set of dimension $r<s$, then there exist a point $\alpha\in T$ and a ray $R=\{\alpha+\lambda\bm{e} \mid \lambda\ge0\}$, such that $R\cap T_i\nsubseteq T$ for any $i$. 
\end{lemma}

\begin{proof}
Possibly taking linear transformations, we may assume that $T\subseteq [-1,1]^r\times \bm{0}_{s-r}$. If $r=0$, then $T$ is a point $T = \bm{0}_{s}$. If $r>0$, then we may further assume that $(\bm{0}_{r},\bm{0}_{s-r})$ is a relative interior point of $T$. 

\medskip

Let $\pi: \Rr^s\to\Rr^{s-r}$ be the projection $(x_1,\ldots,x_s)\mapsto(x_{r+1},\ldots,x_s)$. Then $\pi(T)=\bm{0}_{s-r}\in\Rr^{s-r}$, and $\{\pi(T_i)\}_{i\in\Nn}$ is a decreasing sequence of compact polytopes in $\Rr^{s-r}$. Let $C_{\bm{0}_{s-r}}(\pi(T_i)) \subseteq \Rr^{s-r}$ be the cone generated by $\pi(T_i)$ with vertex $\bm{0}_{s-r}$. 

\medskip

Let $H$ be a hyperplane in $\Rr^{s-r}$, such that $\bm{0}_{s-r}\notin H$ and $H$ is not parallel to any face of $\pi(T_i)$. Then $C_{\bm{0}_{s-r}}(\pi(T_i))\cap H \neq \emptyset$ is compact for any $i$. Since $\{C_{\bm{0}_{s-r}}(\pi(T_i))\cap H\}_{i\in\Nn}$ is a decreasing sequence, $\cap_{i=1}^{\infty} (C_{\bm{0}_{s-r}}(\pi(T_i))\cap H)$ is nonempty. Choose $z\in\cap_{i=1}^{\infty} (C_{\bm{0}_{s-r}}(\pi(T_i))\cap H)$, then for any $i$, there exists $c_i>0$ such that $c_iz\in\pi(T_i)$. Hence there exists $y_i\in\Rr^r$, such that $(y_i, c_iz)\in T_i$ (when $r=0$, this is just $c_iz \in T$). Note that by assumption, there exists $\epsilon_i>0$, such that $(-\epsilon_i y_i,\bm{0}_{s-r})\in T\subseteq T_i$. Since $T_i$ is convex, $$(\bm{0}_r,\frac{\epsilon_ic_i}{\epsilon_i+1}z) = \frac{\epsilon_i}{1+\epsilon_i}(y_i,c_iz) + \frac{1}{1+\epsilon_i}(-\epsilon_i y_i,\bm{0}_{s-r})\in T_i.$$ Now as $(\bm{0}_r,\frac{\epsilon_ic_i}{\epsilon_i+1}z) \not\in T$, the ray $R=\{\bm{0}_s+\lambda z \mid \lambda\ge0\}$ satisfies the requirement.
\end{proof}

\begin{definition}\label{def: unstable point}
	A point $\beta \in \Rr^s$ is called an \emph{unstable point} for a sequence of polytopes $\{P_i\}_{i\in \Nn}$, if for any open set $U \ni\beta$, the set $\{P_i\cap U\mid i\in \Nn\}$ contains infinitely many elements. Otherwise, $\beta$ is a \emph{stable point} for the sequence.
\end{definition}

\begin{lemma}\label{le: DCC implies unstable point}
	Suppose that $\{P_i\}_{i\in\Nn}$ is a strictly increasing sequence of LCT-polytopes satisfying the conditions in Theorem \ref{thm: ACC for LCT-polytopes}. Then there exists at least one unstable point $\beta \in P_i$ for $i \gg 1$.
\end{lemma}

\begin{proof}
	 If there is no unstable point on $\cup P_i$, then every point $x\in \cup P_i$ has a neighborhood $U_x$ such that $\{P_i\cap U_x \mid i\in\Nn\}$ has finitely many elements. By Lemma \ref{le: compactness of union}, $\cup_{i\geq 1} P_i$ is compact. Thus there exists a finite subcover of $\{U_x\mid x\in\cup P_i\}$, which implies that $\{P_i\}_{i\in\Nn}$ is not strictly increasing, a contradiction.
\end{proof}

\begin{lemma}\label{lem: lctfacecone}
	Let $s\ge2$, and let $P=P(X,\Delta; D_{1},\ldots,D_{s})$ be an LCT-polytope of dimension $s$ in $\Rr^s$. Suppose that $\beta=(b_1,\ldots,b_s)$ is a point on the boundary of $P$. Let $\pi:\Rr^s\to \Rr^{s-1}$ be the projection  $(x_1,\ldots,x_s)\mapsto (x_1,\ldots,x_{s-1})$, and
	\[
	\BB:=\{F \mid \beta\in F\text{ is a facet of }P \text{~and~}F\nsubseteq \{(x_1,\ldots,x_s) \mid x_s=0\}\}.
	\] If $\dim C_{\pi(\beta)}(\pi(\BB))=s-1$, then $C_{\pi(\beta)}(\pi(\BB))=C_{\pi(\beta)}(\pi(P))$, where $\pi(\BB) \coloneqq \cup_{F\in \BB}\pi(F)$.
\end{lemma}
\begin{proof}
	By definition, $\pi(\mathfrak{B})\subseteq \pi(P)$. Suppose to the contrary that there exists a point $\alpha'\in \pi(P)\setminus C_{\pi(\beta)}(\pi(\BB))$. Let $L$ be the line segment with endpoints $\alpha'$ and $\pi(\beta)$. Then $L\subseteq \pi(P)$ and $L\cap C_{\pi(\beta)}(\pi(\BB))=\{\pi(\beta)\}$. 
	Let \begin{equation*}
	\BB'=\{F \mid F\text{ is a facet of }P \text{ and }F\nsubseteq \{(x_1,\ldots,x_s) \mid x_s=0\}\}. 
	\end{equation*} 
We claim that $\pi(\BB')=\pi(P)$. Let $F_s$ be the facet of $P$, such that $F_s\subseteq\{(x_1,\ldots,x_s) \mid x_s=0\}$. It suffices to show that $\pi(F_s)\subseteq \pi(\BB')$. Since the dimension of $P$ is equal to $s$, there exist $\gamma=(\alpha_1,\ldots,\alpha_{s-1})$ and $\alpha_s>0$, such that $(\gamma,\alpha_s)\in P$. 
For any $\gamma'\in\pi(P)$, if $(\gamma',0)$ belongs to the boundary of $F_s$, then there exists another facet $F_s'$ of $P$, such that $(\gamma',0)\in F_s'$, thus $\gamma'\in\pi(F_s')\subseteq \pi(\BB')$. If $(\gamma',0)$ is a relative interior point of $F_s$, then there exists $\epsilon>0$, such that $((1+\epsilon)\gamma'-\epsilon\gamma,0)\in F_s$. By the convexity of $P$, $$(\gamma',\frac{\epsilon\alpha_s}{1+\epsilon})=\frac{\epsilon}{1+\epsilon}(\gamma,\alpha_s)+\frac{1}{1+\epsilon}((1+\epsilon)\gamma'-\epsilon\gamma,0)\in P.$$
Thus $\alpha_s'\coloneqq\sup\{t \mid (\gamma',t)\in P\}\ge \frac{\epsilon\alpha_s}{1+\epsilon}>0$, and $(\gamma',\alpha_s')\in F_s'\neq F_s$ for some facet $F_s'$ of $P$. Hence $\gamma'\in \pi(\BB')$. The claim is proved. In particular, $L\subseteq \pi(\BB')$. Since $\pi(\BB')$ is a union of finitely many closed convex subsets $\pi(F)$, there exists a facet $F'\in \BB'$, such that $\{\pi(\beta)\}\subsetneq \pi(F')\cap L$. Since $C_{\pi(\beta)}(\pi(\BB))\cap L=\{\pi(\beta)\}$, $F'\notin\BB$. Let $\beta'=(b_1,\ldots,b_{s-1},b_s')\in F',$ and $\beta'\neq\beta$. Let $H'$ be the hyperplane containing $F'$. Since $\beta\notin F'$, $H'$ is not parallel to $x_s$-axis. By Lemma \ref{lem: LCTfacet}, the connected component of $\Rr^s\backslash H'$ containing $P$ is $\{(x_1,\ldots,x_s)\mid a_1'x_1+\cdots+a_s'x_s\le a_0'\},$ where $a_i'\ge0$ and $\sum_{i=1}^{s-1} a_i'b_i+a_s'b_s'=a_0'$. Since $\beta\in P$ and $\beta\notin F'$, $a_s'>0$ and
    $\sum_{i=1}^{s} a_i'b_i< a_0'$. Thus $b_s'>b_s$. 
    
    \medskip
    
For any $F\in \BB$, let $H$ be the hyperplane containing $F$. By Lemma \ref{lem: LCTfacet}, the connected component of $\Rr^s\backslash H$ which contains $P$ is given by either $\{(x_1,\ldots,x_s)\mid x_i\ge 0\}$ for some $i\neq s$, or $\{(x_1,\ldots,x_s)\mid a_1x_1+\cdots+a_nx_n\le a_0\}$ with $a_i\ge0$. In the first case, $F$ is parallel to the $x_s$-axis. In the latter case, since $\beta, \beta' \in P \subseteq \{(x_1,\ldots,x_s)\mid a_1x_1+\cdots+a_sx_s\le a_0\}$, $\beta \in F$ and $b_s'>b_s$, we have $\sum_{i=1}^{s-1}a_ib_i+a_sb_s = a_0$, and $\sum_{i=1}^{s-1}a_ib_i+a_sb_s'\le a_0$. Thus $a_s=0$, and $F$ is also parallel to the $x_s$-axis. Hence $\dim \pi(F) = \dim F -1 =s-2$, and $\dim C_{\pi(\beta)}(\pi(\BB))=s-2$, a contradiction. 
\end{proof}

\begin{proof}[Proof of Theorem \ref{thm: ACC for LCT-polytopes}]
	\noindent We prove the theorem by induction on $s$. If $s=1$, then the theorem follows from the ACC for log canonical thresholds (Theorem \ref{thm: ACC for lct}).
	
	\medskip
	
	Now assume that $s>1$ and Theorem \ref{thm: ACC for LCT-polytopes} holds for less than $s$ testing divisors. Suppose to the contrary that there exists a strictly increasing sequence of LCT-polytopes $\{P_i\coloneqq P(X_i,\Delta_i; D_{i,1},\ldots,D_{i,s})\}_{i=1}^{\infty}$. When $\dim P_i\le s-1$ for any $i$, possibly reordering the coordinates of $\Rr^s$, we may assume that $P_i\subseteq\{(x_1,\ldots,x_s)\mid x_s=0\}$. In fact, if $(a_1,\ldots, a_s) \in P_i$ with $a_j>0$ for all $1\leq j\leq s$, then $[0,a_1] \times \cdots \times [0,a_s] \subseteq P_i$, and thus $\dim P_i=s$. Hence, $\{P_i=P(X_i,\Delta_i; D_{i,1},\ldots,D_{i,s-1})\}_{i=1}^{\infty}$ is a strictly increasing sequence of LCT-polytopes with $s-1$ testing divisors. This contradicts the induction hypothesis. Hence, possibly passing to a subsequence of $\{P_i\}$, we may assume that $\dim P_i=s$ for any $i$. By Lemma \ref{le: DCC implies unstable point}, the sequence $\{P_i\}$ has an unstable point $\beta=(b_1,\ldots,b_s)$. Possibly passing to a subsequence of $\{P_i\}$, we may assume that $\beta$ lies on the boundary of $P_i$.
    
\medskip
	
For a facet $F_i' \subseteq P_i$, let $H_i'\subseteq \Rr^s$ be the hyperplane containing $F_i'$. Since $\beta$ is an unstable point, possibly passing to a subsequence of $\{P_i\}$, there exists a facet $\beta \in F_i'$ of $P_i$, such that $H_i'\notin\{H_1',\ldots,H_{i-1}'\}$ for each $i$. 

\medskip 

Possibly passing to a subsequence of $\{P_i\}$ and reordering the coordinates, we may assume that for any $i$, $H_i'$ is not parallel to the $x_s$-axis and $H_i'\neq \{(x_1,\ldots,x_s)\mid x_s=0\}$. Thus the hyperplane $H_i'$ is defined by a linear equation $x_s=t_i'(x_1,\ldots,x_{s-1})$, where $t_i'$ is a nonzero linear function of $x_1,\ldots,x_{s-1}$. 
	
\medskip
	
Let $\pi:\Rr^s\to \Rr^{s-1}$ be the projection $(x_1,\ldots,x_s)\mapsto (x_1,\ldots,x_{s-1})$, and  $\BB_i=\{F \mid \beta\in F\text{ is a facet of }P_i \text{ and }F\nsubseteq \{(x_1,\ldots,x_s) \mid x_s=0\}\}$. By construction, $F_i'\in\BB_i$, so $\BB_i \neq \emptyset$ and $\dim \pi (\BB_i)=s-1$. Consider the following two cases.

\medskip	

\noindent {\bf Case 1.} There exists an index $j_0$ and infinitely many $i$, such that $\pi(\BB_{j_0})\cap \pi(\Int(F_i'))\neq\emptyset$.
	
	\medskip
	
Passing to a subsequence of $\{P_i\}$, we may assume that $j_0=1$ and $\pi(\BB_{1})\cap \pi(\Int(F_i'))\neq\emptyset$ for any $i$. Since $\dim \pi(\BB_{1}) = s-1$, there is a facet $F_1\in\BB_1$ such that $T_1\cap \pi(F_i')$ is a polytope of dimension $s-1$ for any $i$, where $T_1:=\pi(F_{1})$. Then $\{\pi^{-1}(T_1)\cap P_i\}_{i=1}^{\infty}$ is a strictly increasing sequence of polytopes.

	\medskip
	
	We prove by induction that for any $j\geq 1$, possibly passing to a subsequence of $\{P_i\}$, there exists a facet $F_j\in \BB_j$, such that $T_j\cap \pi(F_i')$ is a polytope of dimension $s-1$ for any $i\ge j$, where $T_j \coloneqq \cap_{i=1}^{j}\pi(F_j)$. In particular, $\{\pi^{-1}(T_j)\cap P_i\}_{i=j}^{\infty}$ is a strictly increasing sequence of closed polytopes. 
	
	\medskip
	
	The case where $j=1$ has been proved. When $j\ge2$, by induction, $C_{\pi(\beta)}(T_{j-1})\cap \Int(\pi(F_i'))$ is of dimension $s-1$ for any $i\ge j-1$. By Lemma \ref{lem: lctfacecone}, 
	\[
	\begin{aligned}
	C_{\pi(\beta)}(T_{j-1})\subseteq C_{\pi(\beta)}(\pi(\BB_{j-1}))&=C_{\pi(\beta)}(\pi(P_{j-1}))\\
	&\subseteq C_{\pi(\beta)}(\pi(P_j))=C_{\pi(\beta)}(\pi(\BB_j)).
	\end{aligned}
	\] 
	Thus passing to a subsequence of $\{P_i\}_{i\geq j}$, there exists a facet $F_j\in \BB_j$, such that $C_{\pi(\beta)}(\pi(F_j))\cap C_{\pi(\beta)}(T_{j-1})\cap \Int(\pi(F_i'))$ is of dimension $s-1$ for any $i \geq j$. Hence $T_j\cap \Int(\pi(F_i'))=\pi(F_j)\cap T_{j-1}\cap\Int(\pi(F_i'))\neq\emptyset$, and $T_j\cap \pi(F_i')$ is a polytope of dimension $s-1$ for any $i\ge j$. This finishes the induction. 
	
	\medskip

	Let $H_j$ be the hyperplane containing $F_j$, and let 
	\begin{equation*}\label{eq: tj}
	x_s=t_j(x_1,\ldots,x_{s-1})
	\end{equation*} 
	be the equation of $H_j$. Then $t_j(x_1,\ldots,x_{s-1})\ge t_{j-1}(x_1,\ldots,x_{s-1})$ for any $(x_1,\ldots,x_{s-1})\in T_j$. Let $T\coloneqq\cap_{i=1}^{\infty} T_i$, then $\pi(\beta)\in T$ and $T$ is a closed convex set. 	
	
	\medskip
  
	If $\dim T=s-1$, then choose $\alpha_1,\ldots,\alpha_s\in T$, such that $\alpha_1-\alpha_s,\ldots,\alpha_{s-1}-\alpha_s$ are $\Rr$-linearly independent. For $(x_1,\ldots,x_{s-1})\in \pi(P_i)$, set 
	\begin{equation*}\label{eq: ts}
	t_i^{s}(x_1,\ldots,x_{s-1})\coloneqq\sup\{x_s \mid (x_1,\ldots,x_s)\in P_i\}.
	\end{equation*} 
	
	For any $1\le k\le s$ and $i \geq 1$, $t_i(\alpha_k)=t_i^{s}(\alpha_k)$. Since for any fixed $1 \leq k \leq s$, $\{t_i^s(\alpha_k)\}$ is increasing with respect to $i$, by Theorem \ref{thm: ACC for lct}, passing to a subsequence of $\{P_i\}$, we may assume that $t_2^{s}(\alpha_k)=t_1^{s}(\alpha_k)$ for any $1\le k\le s$. Thus $t_2(\alpha_k)=t_1(\alpha_k)$ for any $1\le k\le s$. Since $t_j(x_1,\ldots,x_{s-1})$ are linear functions of $s-1$ variables, $t_2(x_1,\ldots,x_{s-1})\equiv t_1(x_1,\ldots,x_{s-1})$. Recall that $P_2$ is an LCT-polytope, $t_2^{s}(\alpha')\le t_2(\alpha')$ for any $\alpha'\in T_1 \subseteq \pi(P_2)$ with the equality holds if and only if $\alpha'\in \pi(F_2)$. Since $P_1\subseteq P_2$ and $T_1=\pi(F_1)$,
	\[
	t_1(\alpha')=t_1^{s}(\alpha')\le t_2^{s}(\alpha')\le t_2(\alpha')=t_1(\alpha').
	\]
	So all the equalities hold. In particular, $\alpha'\in \pi(F_2)$ and $T_1\subseteq \pi(F_2)$. This implies that $\pi^{-1}(T_{1})\cap P_1=\pi^{-1}(T_{1})\cap P_2$, which is a contradiction since $\{\pi^{-1}(T_1)\cap P_i\}_{i=1}^{\infty}$ is strictly increasing.
	
	\medskip

	If $\dim T\le s-2$, then by Lemma \ref{lem:commona ray}, there exists a point $\alpha\in T$, and a ray $R=\{\alpha+\lambda\bm{e} \mid \lambda\ge0\}$, such that $R\cap T_i\nsubseteq T$ for any $i$. 
	Let 
	\[
	t_i|_{R}(\lambda)\coloneqq t_i(\alpha+\lambda\bm{e}).
	\]
	By Theorem \ref{thm: ACC for Local LCT-polytopes}, $t_i|_{R}(\lambda)$ belongs to a finite set. Possibly passing to a subsequence, we may assume that 
	$t_i|_{R}(\lambda)=t_1|_{R}(\lambda)$ for any $i$. We will show that
	$R\cap T_1\subseteq\pi(F_i)$ for any $i$. Let $\alpha'=\alpha+\lambda'\bm{e}$ be any point on $R \cap T_1$, and 
	$b_{i,s}':=t_i^{s}(\alpha')$. Since $P_i$ is an LCT-polytope, 
	\[
	b'_{i,s}\le t_i|_{R}(\lambda')=t_1|_{R}(\lambda'),
	\]
	with the first equality holds if and only if $(\alpha',b_{i,s}')\in F_i$. Recall that $T_1=\pi(F_1)$ and $\{P_i\}$ is increasing, we have $t_1^{s}(x_1, \ldots, x_{s-1})=t_1(x_1, \ldots, x_{s-1})$ on $T_1$, and $t_1|_{R}(\lambda')=t_1^s(\alpha')\le t_i^s(\alpha') =b_{i,s}'$. Thus $b_{i,s}'=t_i|_{R}(\lambda')$ and $\alpha'\in \pi(F_i)$. Hence $R \cap T_1 \subseteq \pi(F_i)$ for any $i$. Therefore, 
	$$R\cap T_1\subset \cap_{i=1}^{\infty}\pi(F_i)=\cap_{i=1}^{\infty}T_i=T,$$
	which contradicts the choice of $R$.

\medskip	
	
\noindent {\bf Case 2}. Now assume that for each $j$, $\pi(\BB_j) \cap \pi(\Int(F_i')) = \emptyset$ for any $i\gg 1$. Possibly passing to a subsequence, we may assume that  $\pi(\BB_{i-1})\cap \pi(\Int(F'_{i}))=\emptyset$ for any $i\ge 2$. Thus $C_{\pi(\beta)}(\pi(\BB_{i-1}))\cap C_{\pi(\beta)}(\Int(\pi(F'_{i})))=\{\pi(\beta)\}$.
        
        \medskip
        
Let $y_{i}\in \Int(\pi(F'_{i}))\setminus C_{\pi(\beta)}(\pi(\BB_{i-1}))$. Recall that $\dim \pi(\BB_{i-1})=s-1$, by Lemma \ref{lem: lctfacecone}, 
\[
    C_{\pi(\beta)}(y_i)\cap C_{\pi(\beta)}(\pi(P_{i-1}))=C_{\pi(\beta)}(y_i)\cap C_{\pi(\beta)}(\pi(\BB_{i-1}))=\{\pi(\beta)\}.
\] 

Since $y_i\in\pi(F'_{i})$, there exits $z_i \geq 0$ such that $(y_i,z_i)\in P_i$. Recall that $b_s$ is the $x_s$-coordinate of $\beta$. Possibly replacing $(y_i,z_i)$ with an interior point on the line segment connecting $\beta$ and $(y_i,z_i)$ which is close enough to $\beta$, we may assume $z_i\ge \frac{b_s}{2}$. Now $(y_i,\frac{b_s}{2})\in P_i\setminus P_{i-1}$ since $y_i \not\in \pi(\BB_{i-1}) = \pi(P_{i-1})$. Hence
\[
    P_i\cap \{(x_1,\ldots,x_s) \mid x_s=\frac{b_s}{2}\}=P(X_i,\Delta_i+\frac{b_s}{2}D_{i,s}; D_{i,1},\ldots,D_{i,s-1})
\]
are strictly increasing LCT-polytopes with $s-1$ testing divisors, a contradiction. 		
\end{proof}

	The argument below is parallel to that of \cite[Corollary 1.10]{HMX14}.

\begin{proof}[Proof of Corollary \ref{cor: Fano spectrum}]
	Suppose that $\mathcal{R}$ contains an increasing sequence $r_1\le r_2\le \cdots$, such that $-(K_{X_i}+\Delta_i(a_{X_i}))\sim_{\Rr} r_iH_i$ for some Cartier divisor $H_i$ and $(X_i,\Delta_i(t),a_{X_i})\in \mathcal{D}$. By the cone theorem (see \cite[Theorem 18.2]{Fujino11}), there exists a curve $C_i$ such that $-(K_{X_i}+\Delta_i(a_{X_i}))\cdot C_i\le 2n$. In particular, $r_i\le 2n$. By the effective base point free theorem for log canonical pairs (see \cite[Theorem 1.1]{Fujino09}), there is a universal $m \in \Nn$ such that the linear system $|mH_i|$ is base point free. Possibly replacing $m$ by a multiple, we may assume $m>2n$. Pick a general divisor $D_i\in|mH_i|$, then
	\begin{equation*}
		(X_i,\Gamma_i(t))\coloneqq(X_i,\Delta_i(t)+\frac{r_i}{m}\cdot\frac{b-t}{b-a_{X_i}}D_i)
	\end{equation*} 
	is log canonical for $t\in[a_{X_i},b]$. Moreover, as $K_{X_i}+\Gamma_i(b) \sim_{\Rr} K_{X_i}+\Gamma_i(a_{X_i}) \sim_{\Rr} 0$, by linearity,
	\begin{equation*}
		K_{X_i}+\Gamma_i(t)\sim_{\Rr} 0\text{ for }t\in[a, b].
	\end{equation*}

	As $a_{X_i} \in \mathcal B$ and $\mathcal B$ is DCC, $\{\frac{b-a}{b-a_{X_i}}\}_{i\in\Nn}$ is DCC as well. Since $\{r_i\}_{i\in\Nn}$ is increasing, the coefficients of $\{\Gamma_i(a)\}_{i\in\Nn}$ and $\{\Gamma_i(b)\}_{i\in\Nn}$ belong to a DCC set. By Theorem \ref{thm: linear global ACC}, $\{\frac{r_i}{m}\cdot\frac{b-t}{b-a_{X_i}}\}_{i\in\Nn}$ is a finite set. By the DCC property of $\{\frac{1}{b-a_{X_i}}\}$ again, $\{r_i\}_{i\in\Nn}$ cannot be strictly increasing.
\end{proof}

\subsection{Concluding remarks}\label{subsection: concluding remarks}
	We discuss several observations from the proofs and propose some related questions.
	
\medskip

	\cite[Definition 4.1]{BZ16} introduced generalized polarized pairs and established the ACC property for such pairs. Our results are speculated to hold in that setting as well.

\medskip

	Since our method is local, we are unable to deal with global problems. For example, it was shown that the accumulation points of log canonical thresholds belong to the set of log canonical thresholds of lower dimensional varieties (see \cite[Theorem 1.1]{MP04} and \cite[Theorem 1.11]{HMX14} for the precise statement). It is not known whether some similar property holds for LCT-polytopes, even for smooth varieties (an accumulation polytope is the limit of infinitely many \emph{different} LCT-polytopes). However, \cite[Theorem 3.3]{LM11} showed that in the smooth case, a sequence of LCT-polytopes converges to an LCT-polytope with respect to the Hausdorff metric.

\medskip

	One potential application of LCT-polytopes might be on the problems related to the existence of K\"ahler-Einstein metrics. Traditionally, the so-called $\alpha$-invariant was introduced by Tian to deal with the existence of K\"ahler-Einstein metrics (see \cite{Tian87}). For a given $\Qq$-Fano variety $X$, the alpha invariant $\alpha(X)$ was shown be $\inf\{\lct(X;D)\mid 0\le D\sim_{\Qq}-K_X\}$ (see \cite[Theorem A.3]{CS08}). More recently, log canonical thresholds also appeared in the study of stabilities of varieties (see \cite{Fuj15, Fujita16} and references therein). It is expected that LCT-polytopes could give some refined description for the existence of K\"ahler-Einstein metrics.

\medskip

	Finally, one can also generalize other invariants to multiple divisors. In \cite{HL20,HanLi20accumulation}, we study the generalization of pseudo-effective thresholds. Under some conditions, we prove the generalized Fujita's spectrum conjecture and their ACC property (see \cite{Fuj92, Fuj96, DiC16, DiC17}, etc.). The $\Rr$-complementary thresholds is a generalization of log canonical thresholds, and it satisfies the ACC as well \cite{HLS19}. It is an interesting question that whether $\Rr$-complementary thresholds polytopes satisfy the ACC or not.

\newcommand{\etalchar}[1]{$^{#1}$}

\end{document}